\documentclass[11pt]{article}
\usepackage[utf8]{inputenc}
\usepackage{amssymb,amsmath,amsthm,amsfonts,amscd,bbm}
\usepackage{hyperref}
\usepackage{xcolor}
\usepackage{graphicx}
\usepackage[outdir=./]{epstopdf}
\usepackage{enumitem}
\usepackage{breqn}
\usepackage{tensor}
\usepackage{tikz-cd}

\newcommand{\googlebooks}[1]{(preview at \href{https://books.google.com/books?id=#1}{google books})}

\newcommand{\floor}[1]{\lfloor #1 \rfloor}
\newcommand{\numdam}[1]{}

\def\semicolon{;}
\def\applytolist#1{
    \expandafter\def\csname multi#1\endcsname##1{
        \def\multiack{##1}\ifx\multiack\semicolon
            \def\next{\relax}
        \else
            \csname #1\endcsname{##1}
            \def\next{\csname multi#1\endcsname}
        \fi
        \next}
    \csname multi#1\endcsname}

\def\calc#1{\expandafter\def\csname c#1\endcsname{{\mathcal #1}}}
\applytolist{calc}QWERTYUIOPLKJHGFDSAZXCVBNM;

\theoremstyle{plain}
\newtheorem{thm}{Theorem}[section]
\newtheorem*{thm*}{Theorem}

\newtheorem{cor}[thm]{Corollary}

\newtheorem*{cor*}{Corollary}

\newtheorem*{conj*}{Conjecture}

\newtheorem*{hypothesis*}{Hypothesis}
\newtheorem*{philosophy*}{Philosophy}
\newtheorem{lem}[thm]{Lemma}

\newtheorem{prop}[thm]{Proposition}

\newtheorem*{quest*}{Question}
\newtheorem*{claim*}{Claim}

\theoremstyle{definition}
\newtheorem{defn}[thm]{Definition}

\newtheorem{ex}[thm]{Example}
\newtheorem{sub-ex}[thm]{Sub-Example}
\newtheorem{rem}[thm]{Remark}
\newtheorem*{rem*}{Remark}

\usepackage{longtable}
\usepackage{fullpage}

\title{Remarks on anomalous symmetries of C*-algebras}
\author{Corey Jones }
\date{}

\begin{document}

\maketitle

\begin{abstract}

For a group $G$ and $\omega\in Z^{3}(G, \text{U}(1))$, an $\omega$-anomalous action on a C*-algebra $B$ is a $\text{U}(1)$-linear monoidal functor between 2-groups $\text{2-Gr}(G, \text{U}(1), \omega)\rightarrow \underline{\text{Aut}}(B)$, where the latter denotes the 2-group of $*$-automorphisms of $B$. The class $[\omega]\in H^{3}(G, \text{U}(1))$ is called the anomaly of the action. We show that for every $n\ge 2$ and every finite group $G$, every anomaly can be realized on the stabilization of a commutative C*-algebra $C(M)\otimes \mathcal{K}$ for some closed connected $n$-manifold $M$. We also show that although there are no anomalous symmetries of Roe C*-algebras of coarse spaces, for every finite group $G$, every anomaly can be realized on the Roe corona $C^{*}(X)/\mathcal{K}$ of some bounded geometry metric space $X$ with property $A$. 
    
\end{abstract}

\section{Introduction.}

C*-algebras are often used to describe the structure of observables in a quantum system. Symmetries of the observable algebra correspond to global symmetries of the physical system. However, the observable algebra may have \textit{anomalous symmetries}, which are physically characterized by their inability to be represented covariantly as unitary operators on any Hilbert space representation of the system \cite{MR903016, MR1324403} (see Section \ref{physicalinterpretation}). In many situations, this type of anomaly can be described by a class $[\omega]\in H^{3}(G,\text{U}(1))$ which we call a \textit{phase anomaly}. Phase anomalies are responsible for a variety of interesting phenomena in quantum field theory and condensed matter physics (for example see Section \ref{symmetrybreaking}).

Abstractly, it is convenient to describe anomalous symmetries of C*-algebras in the language of higher categories. A \textit{2-group} is a monoidal category where all objects and morphisms are invertible. Given a group $G$ and a 3-cocycle $\omega\in Z^{3}(G, \text{U}(1))$, this defines a 2-group $\text{2-Gr}(G, \text{U}(1), \omega)$ (see Section \ref{2GrandAnom} for definitions). Associated with a C*-algebra $B$ is the $\text{U}(1)$-linear 2-group of *-automorphisms $\underline{\text{Aut}}(B)$. An \textit{$\omega$-anomalous action} of $G$ on $B$ is then defined to be a $\text{U}(1)$-linear monoidal functor $\text{2-Gr}(G, \text{U}(1), \omega)\rightarrow \underline{\text{Aut}}(B)$. The class $[\omega]\in H^{3}(G, \text{U}(1))$ is called the \textit{anomaly} of the symmetry. We can unpack this categorical definition to obtain an elementary description as follows:

\begin{defn}\label{Explicitdefn} Let $G$ be a group and $\omega\in Z^{3}(G, \text{U}(1))$. An $\omega$-anomalous action on a C*-algebra $B$ is an assignment of a $*$-automorphism $\alpha_{g}\in \text{Aut}(B)$ for each $g\in G$, and for each pair $g,h\in G$, a unitary $m_{g,h}\in M(B)$ subject to the following conditions:

\begin{enumerate}
    \item  $m_{g,h}\alpha_{g}(\alpha_{h}(x))=\alpha_{gh}(x)m_{g,h}$ for all $x\in B$.
    \item
    $\omega(g,h,k)m_{gh,k}m_{g,h}=m_{g,hk}\alpha_{g}(m_{h,k})$.
\end{enumerate}
\end{defn}

The study of anomalous symmetries in the theory of operator algebras is not new \cite{MR394228, MR448101, MR548118, Jo1980, Su1980, MR1324403, MR1798596}. Anomalies have appeared in this context as obstructions to the existence of a twisted (also called ``cocycle") action of $G$ on $B$ that lifts a homomorphism $G\rightarrow \text{Out}(B)$\footnote{These homomorphisms are called G-kernels in the literature. See Section \ref{physicalinterpretation} for a more detailed explanation of the relationship between anomalies and lifting obstructions.}. Our point of view centered on 2-group actions fits into the more general higher categorical approach to operator algebra symmetries using correspondences \cite{MR3056650, MR2986862}.  Examples of anomalous actions of groups on C*-algebras arise from general constructions of unitary tensor category actions on C* and von Neumann algebras that have emerged from subfactor theory \footnote{See Remark \ref{relationwithfusion} for a discussion of the relationship between anomalous actions and actions of pointed unitary fusion categories on C*-algebras.} (e.g. \cite{Oc1988, Iz1993, MR1334479, Iz1998, LoRo1997, Mg2003, MR2732052, MR4166625, MR3959065, MR4139893, MR3935804, MR3663592}).

Many of the C*-algebras that are used in physical applications are closely related to classical topology and geometry. For example, algebras arising from bundles of matrix algebras over manifolds appear in noncommutative approaches to the Standard Model (e.g. \cite{MR1463819, MR1303779}), and in the theory of T-duality \cite{MR2116734, MR2560910}. Roe C*-algebras of metric spaces spaces have found applications in the theory of topological band structures in condensed matter physics \cite{MR3927086, MR3594362}.
This motivates us to study anomalous symmetries for these classes of C*-algebras, since they may have interesting physical interpretations. However, the most well known constructions of anomalous symmetries of C*-algebras, using ideas from subfactor theory, produce algebras which are more closely related to combinatorial structures like finite graphs than to topological or geometric spaces. We therefore require different approaches for building anomalous group actions on wider classes of C*-algebras in order to address basic existence questions.

In this paper, we adapt the algebraic techniques originally introduced by Eilenberg and Maclane \cite{MR19092, MR20996, MR33287} to build $\omega$-anomalous actions on crossed product C*-algebras using ordinary group actions together with certain cohomological data (see Theorem \ref{constructingtheaction}). These ideas have already been used effectively to build anomalous actions of groups on von Neumann algebras \cite{MR548118, Su1980}. They play a central role in the existence portion of V. Jones' classification of finite group actions on the hyperfinite $\rm{II}_{1}$ factor $R$, a result which implies all anomalies of all finite groups are realizable on $R$ in an essentially unique way \cite{Jo1980}. 

While the construction provided by Theorem \ref{constructingtheaction} is widely applicable, we focus on two main classes of C*-algebras: commutative C*-algebras $C(X)$ and their stabilizations, where $X$ is a compact connected Hausdorff space; and Roe C*-algebras of bounded geometry metric spaces and their quotients, which play an important role in coarse geometry. In both cases, we find some obstructions to the existence of anomalous actions, and conversely, general constructions that show anomalous actions are quite common.

First we consider the case of commutative C*-algebras. It is easy to find examples of anomalous finite group actions on the commutative C*-algebra $C(X)$ for $X$ a finite discrete space: finitely semi-simple module categories over the pointed unitary fusion categories $\text{Hilb}(G, \omega)$ provide natural examples, with $X$ the finite set of isomorphism classes of simple objects. It is more difficult to find examples with connected spectrum $X$, even after passing to the stabilization $C(X)\otimes \mathcal{K}$, where $\mathcal{K}$ is the C*-algebra of compact operators on a separable Hilbert space. We have the following no-go theorem, which gives topological obstructions to the existence of anomalous actions:

\begin{thm}\label{noactionsonmanifolds}
Let $X$ be a compact connected, locally path connected Hausdorff space. If $H^{1}(X, \mathbbm{Z})=0$,\footnote{$H^{k}(X, \mathbbm{Z})$ of a topological space $X$ will always denote singular cohomology in this paper.} then there are no anomalous actions on $C(X)$ for any finite group $G$. If in addition $X$ has no non-trivial complex line bundles\footnote{If $X$ has the homotopy type of a CW-complex, this is equivalent to $H^{2}(X, \mathbbm{Z})=0$} (e.g. homology spheres of dimension $n\ge 3$), there are no anomalous actions of any finite groups on the stabilization $C(X)\otimes \mathcal{K}$.
\end{thm}

In Section \ref{anomaliesintopology} we construct several examples of anomalous finite group actions on commutative C*-algebras with compact connected spectrum, which fail to satisfy the hypotheses of the above theorem (See Sections \ref{ActionsonS} and \ref{actionon2torus}). We then apply the constructions from Section \ref{buildingactions} to obtain the following theorem:

\begin{thm}\label{actionsonmanifolds}
For every finite group $G$, every $\omega\in Z^{3}(G, \text{U}(1))$, and every $n\ge 2$, there exists a closed connected n-manifold $M$ and an $\omega$-anomalous action of $G$ on $C(M)\otimes \mathcal{K}$. For $n\ge 4$, $M$ can be chosen so that $H^{1}(M, \mathbbm{Z})=0$. 
\end{thm}

A comparison of these theorems shows that the anomalies appearing in the latter theorem must be using the noncommutativity afforded by stabilization in a non-trivial way, since the first result shows they cannot arise from actions on the commutative C*-algebra $C(M)$ alone. In Corollary \ref{non-trivial-lifting}, we leverage this to show for every $n\ge 4$ there are infinitely many connected $n$-manifolds $M$ such that the lifting obstruction $[\omega_{M}]\in H^{3}(\text{Out}(C(M)\otimes \mathcal{K}), C(X, \text{U}(1)) )$ is non-trivial.

C*-algebras play an important role in coarse geometry \cite{MR2007488}. For a metric space $X$ with bounded geometry, the Roe C*-algebra $C^{*}(X)$ captures a significant amount of information about the large scale structure of $X$. Restricting to spaces with property A, these algebras are a complete invariant of the coarse equivalence class of the space \cite{MR3116573}. If $X$ has property $A$ then $\text{Out}(C^{*}(X))\cong \text{Coa}(X)$, the group of coarse auto-equivalences up to the relation of ``closeness" \cite{BV2020}. This suggests that anomalies for a finite group acting by coarse autoequivalences on a coarse space should have some sort of coarse-geometric interpretation. However, we have the following no-go theorem:

\begin{thm}\label{noactiononcoarse}
Let $X$ be a discrete metric space with bounded geometry. Then there are no anomalous actions of any group on the Roe algebra $C^{*}(X)$. 
\end{thm}

Recall, however, that the C*-algebra $C^{*}(X)$ contains a unique minimal ideal isomorphic to the algebra $\mathcal{K}$ of compact operators on a separable Hilbert space. The presence of this algebra as the unique minimal ideal is the primary obstruction to the existence of anomalous symmetries. The C*-algebra $C^{*}(X)/\mathcal{K}$ is called the \textit{Roe corona} \cite{BFV2018} due to its similarity to the corona of a non-unital C*-algebras. In Section \ref{anomaliesincoarsegeom} we will show that Roe coronas have plenty of anomalous symmetry. We have the following result:

\begin{thm}\label{actiononPropertyA}
For every finite group $G$ and $\omega\in Z^{3}(G, \text{U}(1))$, there exists a discrete metric space $X$ with bounded geometry and property $A$, and an $\omega$-anomalous action of $G$ on the Roe corona $C^{*}(X)/\mathcal{K}$.
\end{thm}


The outline of the paper is as follows. Section \ref{background} contains some background on group cohomology as well as the definition of 2-group and anomalous action. Section \ref{buildingactions} contains the main construction of anomalous actions we use throughout the paper. Section \ref{anomaliesintopology} focuses on examples from topology, while Section \ref{anomaliesincoarsegeom} discusses anomalies in coarse geometry.

\bigskip

\textbf{Acknowledgements.} This paper is inspired by the work of Vaughan Jones \cite{MR548118, Jo1980}. We dedicate it to his memory. We would like to thank Samuel Evington, Sergio Giron Pacheco, Terry Gannon, Roberto Hernandez Palomares, Mark Pengitore, David Penneys, David Reutter, Yuan-Ming Lu, and Stuart White for many comments and discussions that have been useful for this paper. Many thanks to Rufus Willett for discussions on the K-theory of Roe C*-algebras, and to Andr\'e Henriques for pointing out the $n$ in Theorem \ref{actionsonmanifolds} can be dropped from $3$ to $2$. This research was supported by DMS NSF Grant 2100531/1901082.

\section{Anomalous actions: background and definitions.}\label{background}

We assume the reader is familiar with the basics of categories and C*-algebras. In this section, we recall some background material on cohomology and 2-groups. We also introduce definitions and discuss physical interpretations concerning anomalous actions.

\subsection{Group cohomology.}\label{groupcohomology}

Group cohomology plays an important role in higher category theory. In this paper, we will only make use of the lowest rungs of the higher categorical ladder, and thus will only need the lowest cohomology groups (second and third). Nevertheless, we will include general definitions, primarily to set conventions. In this paper, abelian group operations will be denoted additively with $+$ with one very notable exception: the group $\text{U}(1)$ of unitary complex numbers with multiplication (and any of its natural subgroups) will be written multiplicatively. For groups of functions valued in these groups, we will also use multiplicative notation.

Let $G$ be a group, and let $M$ be a $G$-module. Define 

$$C^{k}(G,M):= \{f:G^{\times k}\rightarrow M\ |\ f(g_{1},\dots, g_{k})=0\ \text{if any}\ g_{i}=1\}.$$

\noindent The elements of $C^{k}(G,M)$ are called \textit{normalized} cochains. Define the differential\footnote{Note that this definition of $d$ may differ by a sign from other conventions.} $d^{k}:C^{k}(G,M)\rightarrow C^{k+1}(G, M)$ by 

$$d^{k}f(g_{1}, \dots, g_{k+1})=$$
$$-g_{1}(f(g_{2},\dots, g_{k+1})+ \sum^{k}_{i=1}(-1)^{i+1}f(g_{1}, \dots, g_{i}g_{i+1}, \dots, g_{k+1})+(-1)^{k}f(g_{1},\dots, g_{k}).$$

We note that $d^{k+1}d^{k}=0$. We often drop mention of the superscript, assuming it is understood from context. Then set $Z^{k}(G, M):=\text{Ker}(d^{k})$, $B^{k}:=\text{Im}(d^{k-1})$ and $H^{k}(G,M):=Z^{k}(G,M)/B^{k}(G,M)$. The $Z^{k}$ are called \textit{normalized cocycles}, $B^{k}$ are called \textit{normalized coboundaries}, and $H^{k}(G, M)$ is called the $k^{th}$ cohomology group of $G$ with coefficients in $M$. The un-normalized versions of the above are defined almost precisely the same way, except the definition of the group of cochains removes the condition that $f(g_{1},\dots, g_{k})=0$ if $g_{i}=1$. The inclusion of the normalized chain complex into the un-normalized version induces an isomorphism on cohomology groups. The equivalence between these approaches is detailed in \cite[Section 6.5]{MR1269324}. The normalized version of group cohomology is certainly more convenient for applications ``in nature", e.g. in group extensions and higher category theory.
When we say cocycle or coboundary in the paper, we will always mean a normalized cocycle or coboundary.

Group cohomology will play an essential role in our story, especially cohomology for finite groups. Standard references for group cohomology include \cite{MR1269324, MR672956}. Given a homomorphism of $G$ -modules $\rho: M\rightarrow N$, we have the pushforward as a (degree 0) map of chain complexes 

$$\rho_{*}: C^{*}(G, M)\rightarrow C^{*}(G, N),$$

\noindent defined for $f\in C^{k}(G, M)$ by

$$\rho_{*}(f)(g_{1},\dots g_{k}):=\rho(f(g_{1},\dots, g_{k})).$$

Similarly, given a homomorphism $\rho: G\rightarrow H$, then any $G$ module $M$ is endowed with the structure of an $H$ module via $h\cdot m:=\rho(h)\cdot m$. The $\rho$ induces a homomorphism of chain complexes called the pullback

$$\rho^{*}: C^{*}(H, M)\rightarrow C^{*}(G, M),$$

$$\rho^{*}(f)(h_{1}, \cdots h_{k}):=f(\rho(h_{1}),\dots, \rho(h_{k})).$$

\noindent The pushforward and pullback are morphisms of chain complexes, hence induce homomorphisms at the level of cocycles, coboundaries, and cohomology.

Given a short exact sequence of groups

\[\begin{tikzcd}
    0 \arrow{r} & L \arrow{r} & M \arrow{r} & N \arrow{r} & 0
 \end{tikzcd},\]

\noindent we have the \textit{long exact sequence in cohomology}

\[\begin{tikzcd}
    \cdots H^{k-1}(G, N) \arrow{r} & H^{k}(G, L) \arrow{r} & H^{k}(G, M) \arrow{r} & H^{k}(G, N) \arrow{r} & H^{k+1}(G, L)\cdots
 \end{tikzcd}\]
 
 \noindent where the arrows between cohomology groups of the same degree are pushforwards. The degree shifting maps $H^{k-1}(G, N)\rightarrow H^{k}(G, L)$ are called \textit{connecting homomorphisms}. We provide an explicit model for these in the proof of Lemma \ref{trivialize3cocycle}. See \cite[Proposition 6.1]{MR672956} for details.

\subsection{2-groups and anomalous actions.}\label{2GrandAnom}

Recall that a monoidal category is a category $\cC$ together with a bifunctor $\cC\times \cC\rightarrow \cC$ (denoted $a\times b\mapsto a\otimes b$), associator isomorphisms $\alpha_{a,b,c}: (a\otimes b)\otimes c\rightarrow a\otimes(b\otimes c)$ satisfying the \textit{pentagon equations}, a unit object $\mathbbm{1}\in \cC$, and unitor isomorphisms $l_{a}: a\otimes \mathbbm{1}\rightarrow a$, $r_{a}: \mathbbm{1}\otimes a \rightarrow a$ satisfying the \textit{triangle equations}. For definitions and details, we refer the reader to Chapter $2$ of the textbook \cite{EtGeNiOs2015}.
Our monoidal categories will typically be \textit{strictly unital}, which means $\mathbbm{1}\otimes a=a=a\otimes \mathbbm{1}$ for all $a\in \cC$ and $l_{a}=1_{a}=r_{a}$

We make extensive use of monoidal functors between monoidal categories, so we include the full definition here for reference [EGNO 2.4.1].

\begin{defn}\label{monfunctor} A monoidal functor between monoidal categories is functor $F:\mathcal{C}\rightarrow \mathcal{D}$ such that $F(\mathbbm{1}_{\cC})\cong 1_{\mathbbm{D}}$, together with isomorphisms $m_{a,b}: F(a)\otimes F(b)\rightarrow F(a\otimes b)$ satisfying

$$F(\alpha^{\mathcal{C}}_{a,b,c})\circ m_{a\otimes b, c}\circ (m_{a,b}\otimes id_{F(c)})= m_{a,b\otimes c}\circ(id_{F(a)}\otimes m_{b,c})\circ\alpha^{\mathcal{D}}_{F(a), F(b), F(c)}. $$

\end{defn}

\noindent An 
\textit{equivalence} between monoidal categories is a monoidal functor which is an equivalence as a functor between categories (i.e. is fully faithful and essentially surjective). This implies there exists a (weak) inverse monoidal functor (see \cite[Section 2.4]{EtGeNiOs2015}).

Recall that a weak inverse of an object $X$ in a monoidal category is an object $X^{-1}$ such that $X\otimes X^{-1}\cong \mathbbm{1}$ and $X^{-1}\otimes X\cong \mathbbm{1}$.

\begin{defn}
A 2-group (also called a Gr-category or categorical group) is a monoidal category such that every object has a weak inverse, and every morphism is an isomorphism.
\end{defn}

A 2-group $\mathcal{G}$ can be described up to equivalence algebraically in terms of its \textit{Postnikov invariants} (see \cite{MR2664619},\cite{EtGeNiOs2015}):

\begin{itemize}
\item 
$\pi_{1}(\mathcal{G})$ is the group of isomorphism class of objects.
\item
$\pi_{2}(\mathcal{G})$ is the group of automorphisms of the tensor unit.
\item
An action of $\pi_{1}(\mathcal{G})$ on $\pi_{2}(\mathcal{G})$.
\item
A (normalized) 3-cocycle $\omega\in Z^{3}(\pi_{1}(\mathcal{G}),\pi_{2}(\mathcal{G})) $, capturing the associator data of the monoidal category.
\end{itemize}

To extract a 3-cocycle from $\mathcal{G}$, for each $g\in \pi_{1}(\mathcal{G})$ pick a representative $X_{g}$. Then for each pair $g,h\in \pi_{1}(\mathcal{G})$ pick an isomorphism $m_{g,h}\in \mathcal{G}(X_{gh},X_{g}\otimes X_{h})$\footnote{For a category $\mathcal{C}$, we use the notation $\mathcal{C}(a,b)$ to denote the set of morphisms from $a$ to $b$}. Then we have

$$(m^{-1}_{g,hk}\circ (1_{g}\otimes m^{-1}_{h,k}))\circ \alpha_{X_{g}, X_{h}, X_{k}}\circ \left( (m_{g,h}\otimes 1_{X_k})\circ m_{gh,k}\right)=\omega(g,h,k)1_{ghk},$$
for some $\omega(g,h,k)\in \pi_{2}(\mathcal{G})$. The pentagon axiom for monoidal categories is equivalent to $\omega\in Z^{3}(\pi_{1}(\mathcal{G}),\pi_{2}(\mathcal{G})) $.
The extraction of the 3-cocycle $\omega$ required arbitrary choices, but any other choices differ by a coboundary, hence the class $[\omega]\in H^{3}(\pi_{1}(\mathcal{G}),\pi_{2}(\mathcal{G})))$ is well defined. Furthermore, if $\mathcal{G}$ is strictly unital and we choose the strict unit $\mathbbm{1}$ to represent $[\mathbbm{1}]$, then choosing $m_{\mathbbm{1}, a}=1_{a}=m_{a,\mathbbm{1}}$, $\omega$ will be a normalized $3$-cocycle\footnote{Every monoidal category is equivalent to a strictly unital one, so we can always make choices to obtain a normalized cocycle}.

Conversely given a triple $(G, A, \omega)$, where $G$ is a group, $A$ is a $G$-module, and $\omega\in Z^{3}(G, A)$ is a (normalized) 3-cocyle, we can build a (strictly unital) 2-group denoted $\text{2-Gr}(G,A,\omega)$ as follows:

\begin{itemize}
\item 
The objects are given by the set $G$, and for $g\ne h, Mor(g,h)=\varnothing$ and $\text{Hom}(g,g)=A$, where we identify the identity in the category $1_{g}$ as the unit in $A$. Thus we write a morphism $a\in \text{Hom}(g,g)$ as $a1_{g}$ 
 \item
Composition of morphisms is the group operation in $A$.
\item
$g\otimes h:=gh$, the group element $e$ is the monoidal unit.

\item
 $(a 1_{g})\otimes (b 1_{h}):=(a g(b)) 1_{gh}$.
\item
The associator in the category $\alpha_{g,h,k}: (g\otimes h)\otimes k=ghk\rightarrow g\otimes (h\otimes k)=ghk$ is given by $\omega(g,h,k) 1_{ghk} $. The unitor is the identity.
     
 \end{itemize} 
 
\noindent Starting with the Postnikov data for a 2-group $\mathcal{G}$ we obtain a natural monoidal equivalence\footnote{we consider monoidal equivalence the natural notion of equivalence of 2-groups} $\mathcal{G}\cong \text{2-Gr}(\pi_{1}(\mathcal{G}),\pi_{2}(\mathcal{G}), \omega)$. The 2-group $\text{2-Gr}(\pi_{1}(\mathcal{G}),\pi_{2}(\mathcal{G}), \omega)$ is a \textit{skeletization} of $\mathcal{G}$.

Now, let $\mathcal{G}_{1}\cong \text{2-Gr}(G_{1}, A_{1}, \omega_{1})$ and $\mathcal{G}_{2}\cong \text{2-Gr}(G_{2}, A_{2}, \omega_{2})$ be Postnikov descriptions for a pair of 2-groups. From the definition of monoidal functor given above, we see a monoidal functor between them is completely described by the following data:

\begin{enumerate}
    \item 
    A homomorphism $\Lambda:G_{1}\rightarrow G_{2}$.
    \item
    A homomorphism $\lambda: A_{1}\rightarrow A_{2}$ such that $\lambda (g(a))= \Lambda(g)(\lambda(a))$ for all $g\in G_{1}, a\in A_{1}$.
    \item
    A co-chain $\mu\in C^{2}(G_1, A_{2})$ such that $d\mu \Lambda^{*}(\omega_{2})=\lambda_{*}(\omega_{1})$.
\end{enumerate}

\noindent where $\Lambda^{*}$ denotes the pullback and $\lambda_{*}$ denotes the pushforward respectively. This will be very useful for us.

There are two specific types of 2-groups that are relevant to this paper. The first are 2-groups of the form $\text{2-Gr}(G,\text{U}(1), \omega)$ where $G$ is a finite group and $\text{U}(1)$ is the group of unitary scalars with trivial $G$-action. These categories arise as the cores of unitary fusion categories, and indeed every pointed fusion category can be obtained uniquely as a linearization of a category of this type.

The other type of 2-group we wish to study is associated to C*-algebras. Let $B$ be a (not-necessarily unital) C*-algebra. Let $\text{Aut}(B)$ denote the group of $*$-autoequivalences of $B$. 

\medskip

We define the 2-group $\underline{\text{Aut}}(B)$ as follows:

\begin{itemize}
    \item 
    Objects are *-automorphisms of $B$, with monoidal product $\alpha\otimes \beta:=\alpha\circ \beta$ given by composition. The monoidal unit is the identity automorphism.
    \item
    A morphism from $\alpha$ to $\beta$ with $\alpha,\beta\in \text{Aut}(B)$ is a unitary element in the multiplier algebra $u\in M(B)$ with $u\alpha(a)=\beta(a)u$ for all $a\in B$.
    \item
    If $u_{i}\in \underline{\text{Aut}}(B)(\alpha_{i},\beta_{i})$ for $i=1,2$ the monoidal product $u_{1}\otimes u_{2}=u_{1}\alpha_{1}(u_{2})$, where the latter uses the unique extension of $\alpha_{1}$ to the multiplier algebra.
    \item
    The associators and unitors are all trivial (i.e. $1\in M(B)$).
    
\end{itemize}

While the above category is strict (trivial associator), there is a great deal of information hidden in the Postnikov invariants. We describe the Postnikov data as follows:

\begin{itemize}
\item
$\pi_{1}(\underline{\text{Aut}}(B))=\text{Out}(B)$, the group of outer automorphisms.
\item
$\pi_{2}(\underline{\text{Aut}}(B))=\text{UZM}(B)$, the group of unitaries in the center of the multiplier algebra of $B$.
\item
The action of $\text{Out}(B)$ on $\text{UZM}(B)$ is the usual action of outer automorphisms on the center.
\item
The 3-cohomology class $[\omega]\in H^{3}(\text{Out}(B), \text{UZM}(B))$ describing the associator is sometimes called a \textit{lifting obstruction}\footnote{From our perspective, this class is an obstruction from lifting a homomorphism $G\rightarrow \text{Out}(B)$ to a $\text{U}(1)$-linear monoidal functor $\text{2-Gr}(G,\text{U}(1),1)\rightarrow \text{\underline{Aut}}(B)$. See also Section \ref{physicalinterpretation}} in the operator algebra literature.
\end{itemize}

In general, the associator is highly non-trivial and contains subtle information about the ``higher'' or ``anomalous" symmetries of $B$. Note that we have a canonical embedding $\text{U}(1)\hookrightarrow \text{UZM}(B)$ for a C*-algebra $B$, and thus all the morphism spaces in the 2-group are canonically $\text{U}(1)$ modules. We have the following definition.

\begin{defn}\label{anomalousaction}
An anomalous action of a group $G$ on a C*-algebra $B$ consists of a 3-cocycle $\omega\in Z^{3}(G, \text{U}(1))$ and a $\text{U}(1)$-linear monoidal functor $\text{2-Gr}(G,\text{U}(1), \omega)\rightarrow \underline{\text{Aut}}(B)$. The coholomogy class $[\omega]\in H^{3}(G, \text{U}(1))$ is called the anomaly.
\end{defn}

An explicit unpacking is given in Definition \ref{Explicitdefn} in the introduction. Note that if we have an isomorphism of C*-algebras $\psi: A\rightarrow B$, then conjugation by $\psi$ gives a monoidal equivalence of $2$-groups $\underline{\text{Aut}}(A)\cong \underline{\text{Aut}}(B)$. If $\omega\in Z^{3}(G, \text{U}(1))$ is given, it will sometimes be convenient to use the terminology $\omega$-anomalous action.

Suppose that we have a homomorphism $\rho: G\rightarrow \pi_{1}(\underline{\text{Aut}}(B))=\text{Out}(B)$. Suppose in addition that the center of the multiplier algebra $\text{ZM}(B)\cong C(X)$ for some compact Hausdorff space $X$. Then pulling back the associator 3-cocycle $\omega_{B}$ gives $\rho^{*}(\omega_{B})\in Z^{3}(G, C(X))$. The question of whether $\rho$ extends to an $\omega$-anomalous action is then equivalent to asking if $i_{*}(\omega)$ is cohomologous to $\rho^{*}(\omega_{B})$ in $Z^{3}(G, C(X, \text{U}(1))$, where $i: \text{U}(1)\rightarrow C(X, \text{U}(1))$ embeds scalars as constant functions. 

This suggests more generally we should consider anomalies with these more general cohomology groups, but we first have to take the action on the center into account. The group $H^{3}(G, \text{U}(1))$ is universal in the sense that it always has a canonical homomorphism to the $H^{3}$ whose coefficients are the topological $G$-modules $C(X, \text{U}(1))$ of interest. This allows us to formulate the question ``does there exists an $\omega$-anomalous G-action on $B$" without having to first specify the spaces $X=\text{Spec}(\text{ZM}(M))$ and the $G$-action on this space, etc. The scalar anomaly also has a physical interpretation as a phase anomaly, discussed in Section \ref{physicalinterpretation}.

\begin{rem}\label{equivalence}
There are (at least) two natural notions of equivalence between $\omega$-anomalous $G$-actions on $B$. The finest is to take  monoidal functors $F_{1}, F_{2}:\text{2-Gr}(G, \text{U}(1), \omega)\rightarrow \underline{\text{Aut}}(B)$ up to monoidal natural isomorphism \cite[Definition 2.4.8]{EtGeNiOs2015}. In general this relation is too fine for classification purposes. We can consider instead the coarser equivalence relation of ``cocycle conjugacy", where we define $F_{1}$ as equivalent to $F_{2}$ if there exists $\alpha\in \text{Aut}(B)$ and a monoidal natural isomomorphism $F_{1}\cong F^{\alpha}_{2}$. Here $F^{\alpha}_{2}$ is defined via pointwise conjugation by $\alpha$, $F^{\alpha}_{2}(g):=\alpha\circ F_{2}(g)\circ \alpha^{-1}$, with structure morphisms given by $\alpha(m^{2}_{g,h})$. We do not study classification type statements in this paper, focusing instead on existence or non-existence questions, thus we do not elaborate these notions.
\end{rem}

\begin{rem}\label{relationwithfusion}
We briefly discuss the relationship between anomalous actions and pointed fusion category actions on a C*-algebra. A unitary fusion category is a finitely semi-simple rigid C*-tensor category \cite{LoRo1997, EtNiOs2005}. An action of a unitary fusion category $\mathcal{C}$ on a C*-algebra is a unitary tensor functor $\mathcal{C}\rightarrow \text{Corr}(B)$, where the latter denotes the C*-tensor category of correspondences \cite{MR3056650}. A fusion category is pointed if all simple objects are weakly invertible. These are always equivalent to fusion categories of the form $\text{Hilb}(G, \omega)$ where $G$ is a finite group $G$ and $\omega\in Z^{3}(G, \text{U}(1))$. 

Embedding $\text{\underline{Aut}}(B)\hookrightarrow \text{Corr}(B)$ as in \cite{MR3056650}, every $\omega$-anomalous action of $G$ on $B$ extends to an action of the fusion category $\text{Hilb}(G, \omega)$ by linearization. Note, however, that using the standard convention of \textit{right} correspondences, this embedding reverses the order of monoidal products (for a detailed explanation, see \cite[Proposition 5.6]{2105.05587}). Conversely a $\text{Hilb}(G, \omega)$ action on $B$ often restricts to an $\omega$-anomalous G-action on $B$, if the correspondences assigned to group elements $G$ are in the image of $\text{\underline{Aut}}(B)\hookrightarrow \text{Corr}(B)$. For finite $G$ this is automatic if, for example, $B$ is a $\rm{II}_{1}$ factor (and everything is normal). However this is not true in general and fails even for commutative C*-algebras (e.g. line bundles). Nevertheless, if $B$ is separable and unital, passing to the stabilization $B\otimes \mathcal{K}$, it is not hard to see that invertible correspondences of $B$ become auto-equivalences of the stabilization.

We also note that for $B$ a C*-algebra with approximate unit of projections, we can consider the unitary Cauchy completion of the projection category of $B$ (i.e. add formal unitary direct sums and formal orthogonal summands). We denote this $\mathcal{C}(B)$. For $B=C(X)$ or $C(X)\otimes \mathcal{K}$, this is the category of finite dimensional Hilbert bundles over $X$, and in general this suggests we should think of these categories as categories of noncommutative vector bundles. Given an $\omega$-anomalous action of $G$ on $B$, then $\mathcal{C}(B)$ acquires the structure of a module category over the unitary fusion category $\text{Hilb}(G, \omega)$. Picking a projection $p\in B$, this gives us a C*-algebra object \textit{internal to} $\text{Hilb}(G,\omega)$ \cite{MR3687214, MR3948170, MR2762528}.

Finally, given a finite $G$ with $\omega\in Z^{3}(G, \text{U}(1))$, the commutative algebra of functions $\text{Fun}(G, \mathbbm{C})$ can be equipped with a quasi-Hopf structure \cite[Example 5.13.6]{EtGeNiOs2015}. It is not hard to see that $\omega$-anomalous action of $G$ on $B$ can be reformulated in terms of co-actions of these quasi-Hopf algebras on $B$.
\end{rem}

\subsection{Gauging and lifting obstructions.}\label{physicalinterpretation}

We will briefly discuss a physical interpretation of anomalous actions, following \cite{MR903016}. Let $B$ be a C*-algebra whose (self-adjoint) operators describe the observables of a physical system. Then automorphisms of $B$ describe global symmetries of the system, while inner automorphisms are typically considered ``gauge" symmetries. It is natural to consider global symmetries up to gauge equivalence, which results in the symmetry group $\text{Out}(B)$. 

Suppose now that we have a homomorphism $\rho:G\rightarrow \text{Out}(B)$. This often called a $G$-kernel by operator algebraists. It is physically relevant to ask whether this symmetry group of the system can be realized covariantly on a Hilbert space, as a necessary step in the process of gauging the symmetry. Gauging the symmetry is a procedure for constructing a new algebra of observables where the G symmetries are now realized by gauge symmetries.

To make mathematical sense of this notion, pick a representative $\alpha_{g}\in \rho(g)$. A \textit{covariant representation} of this choice is a Hilbert space $H$, a non-degenerate representation of $B$ on $H$, and an assignment of unitaries $U_{g}\in U(H)$ for $g\in G$ satisfying:

\begin{enumerate}
    \item 
    The assignment $g\mapsto U_{g}$ is a homomorphism up to gauge equivalence, i.e. there exist $b_{g,h}\in B$ with $U_{g}U_{h}=b_{g,h}U_{gh}$.
    \item
    $U_{g}x=\alpha_{g}(x)U_{g}$ for all $x\in B$.
\end{enumerate}

If we can find such a representation, then associativity of the algebra of operators gives us an expansion of $U_{g}U_{h}U_{k}$ in two different ways:

$$(U_{g}U_{h})U_{k}=b_{g,h}U_{gh}U_{k}=b_{g,h}b_{gh,k}U_{ghk},$$

$$U_{g}(U_{h}U_{k})=U_{g}(b_{h,k}U_{hk})=\alpha_{g}(b_{h,k})b_{g,hk}U_{ghk}.$$

\noindent Taking the inverse of both sides of the equation and setting $m_{g,h}=b^{*}_{g,h}$, the above becomes

\begin{equation}\label{representeqn}
m_{gh,k}m_{g,h}=m_{g,hk}\alpha_{g}(m_{h,k}),\end{equation}
 
\noindent which is precisely the data of a $\text{U}(1)$-linear monoidal functor $\text{2-Gr}(G, \text{U}(1), 1)\rightarrow \underline{\text{Aut}}(B)$.

Conversely, given the above data, we can take the generalized crossed product, pick a state and apply the GNS construction to obtain unitaries acting on a Hilbert space as above.

Note, however, that in quantum mechanics states in a Hilbert space which differ by a phase are physically indistinguishable. This implies that the unitaries one uses in quantum mechanics are also physically indistinguishable if they differ by a phase. Therefore, it is physically reasonable to suppose that instead of equality in the above equation, we have that the two operators $m_{gh,k}m_{g,h}$ and $m_{g,hk}\alpha_{g}(m_{h,k})$ \textit{differ by a phase} i.e. $\omega(g,h,k)m_{gh,k}m_{g,h}=m_{g,hk}\alpha_{g}(m_{h,k})$ for $\omega(g,h,k)\in \text{U}(1)$. It is easy to show this phase gives an element in $Z^{3}(G, \text{U}(1))$. We call this a \textit{phase anomaly} of the symmetry. Furthermore, if we change each of the elements $m_{g,h}$ (and hence the $b_{g,h}$) by a phase, then we change the resulting $3$-cocycle by a coboundary, thus giving us a well-defined element in $H^{3}(G, \text{U}(1))$ which depends on our choices up to a phase. The main point is that while a phase anomaly may be an obstruction to gauging a symmetry, it does not make such symmetries physically irrelevant.

We note that in full generality, for arbitrary choices $b_{g,h}$ satisfying the conditions above, the two sides of equation \ref{representeqn} will differ by a unitary in the center of the multiplier algebra $M(B)$. Thus for observable algebras whose multipliers have non-trivial center, we will obtain, independently of all choices, an element in $H^{3}(G, \text{UZM}(B))$ which is called in the theory of operator algebras the \textit{lifting obstruction} associated to a homomorphism $\rho:G\rightarrow \text{Out}(B)$. This is certainly more natural from a purely mathematical point of view. Indeed this is precisely the pullback under $\rho$ of the Postnikov 3-class (which we previously called the universal lifting obstruction) in $H^{3}(\text{Out}(B), \text{UZM}(B))$. This is non-trivial if and only if there are no representations of unitaries on Hilbert space satisfying the conditions above.

We have the canonical homomorphism $i_{*}:H^{3}(G, \text{U}(1))\rightarrow H^{3}(G, \text{UZM}(B))$ allowing us to compare phase anomalies and lifting obstructions which in general is neither injective nor surjective. Thus phase anomalies and lifting obstructions are not precisely the same thing in the case when the multiplier algebra of $B$ has non-trivial center. If $M(B)$ has trivial center, however, then the lifting obstruction and the homomorphism $\rho:G\rightarrow \text{Out}(B)$ completely determine the phase anomaly of the symmetry. In addition to the quantum mechanical argument that we should only consider unitaries on a Hilbert space up to a phase, a practical reason to specialize to the case of phase anomalies, even in situations where $\text{ZM}(B)$ is non-trivial, is that more general anomalies depend in a detailed way on the action of $G$ on $\text{ZM}(B)$. This makes them difficult to study systematically due to the plethora of finite group actions on commutative C*-algebras.

\subsection{Anomalies force spontaneous symmetry breaking.}\label{symmetrybreaking}

We now discuss the relevance of phase anomalies in the physical context of \textit{spontaneous symmetry breaking}. Spontaneous symmetry breaking occurs when a symmetry of a system (as characterized by compatibility with a Hamiltonian or its associated dynamics), acts non-trivially on a set of states of interest (for example, ground states or equilibrium states a fixed inverse temperature). This is an important mechanism in condensed matter physics for characterizing many types of phase transitions. The idea we present here is that if a symmetry of a system has a non-trivial phase anomaly, this will force symmetry breaking of any \text{pure} state. 

We will keep the discussion fairly general so that it can be applied to many situations of interest. Let $B$ be a unital C*-algebra. Let $\phi$ be a state on $B$. Let $g\mapsto \alpha_{g}\in \text{Aut}(B)$, $m_{g,h}\in U(B)$ be an $\omega$-anomalous action of $G$ on $B$. Then we say this action \textit{preserves} $\phi$ if for all $g\in G$, $x\in B$

$$\phi(\alpha_{g}(x))=\phi(x).$$

As an example, if $B$ has a unique tracial state then all actions preserve this trace. In particular the anomalous actions of $\mathbbm{Z}/n\mathbbm{Z}$ on irrational rotation algebras we construct in Corollary \ref{Irrationalrotation} must preserve the unique tracial state. Thus in general an anomaly is not an obstruction to preserving a state.

However, in many situations of physical interest the states in question are \textit{pure states}, which is equivalent to saying that the corresponding GNS representations are irreducible. This occurs frequently in the study of ground states of quantum statistical mechanical models embodied by C*-dynamical systems \cite{MR1441540, MR1136257}. The extreme points in the convex set of ground states of a (strongly continuous) C*-dynamical system are pure \cite[Theorem 5.3.37]{MR1441540}, and thus for example if there is a unique ground state then it must be pure. Situations where the relevant ground states are pure occur commonly in the study of ground states for gapped Hamiltonians of quantum spin systems such as Kitaev's toric code model (see for example \cite{MR2345476}).

We will show that a pure state cannot be preserved under an anomalous $G$-action, and thus a non-trivial anomaly forces spontaneous symmetry breaking.

\begin{prop}

Let $\phi$ be a pure state on the unital C*-algebra $B$, and suppose that we have an $\omega$-anomalous action of $G$ on $B$. Then if $\phi$ is preserved under this action, $[\omega]\in H^{3}(G,\omega)$ is trivial.
\end{prop}

\begin{proof}

By \cite[Proposition]{2105.05587}, our $\omega$-anomalous action of $G$ on $B$ yields a monoidal functor from the fusion category $\text{Hilb}(G,\bar{\omega})$ to $\text{Corr}(B)$, the monoidal category of (right) B-correspondences (here $\bar{\omega}$ is the pointwise complex conjugate cocycle). We can then equip $\text{Rep}(B)$ with the structure of a $\text{Hilb}(G,\bar{\omega})$ module category in the usual way, by tensoring a representation over $B$ with the appropriate correspondence.

Now, since $\phi$ is fixed under the $G$ action, it is easy to see that the GNS representation $H_{\phi}$ is invariant under the action of simple objects of $\text{Hilb}(G,\bar{\omega})$ up to isomorphism. In particular, if we take $\mathcal{C}$ to be the full subcategory of $\text{Rep}(B)$ whose objects are isomorphic to a finite direct sum of $H_{\phi}$'s, then $\mathcal{C}$ is invariant under the $\text{Hilb}(G,\bar{\omega})$ action. But since $\phi$ is pure, $H_{\phi}$ is irreducible so $\mathcal{C}$ is unitarily equivalent to the category of finite dimensional Hilbert spaces. Thus $\mathcal{C}$ is a rank one module category for $\text{Hilb}(G,\bar{\omega})$, which implies $[\bar{\omega}]=[\omega]^{-1}=[1]$ in $H^{3}(G,\text{U}(1))$ by the classification of finitely semi-simple module categories \cite[Theorem 3.1]{MR1976233}, and thus $[\omega]$ is trivial.
\end{proof}

\section{Anomalous actions on twisted crossed products.}\label{buildingactions}

In this section we describe a method for constructing anomalous actions on twisted crossed product C*-algebras. The underlying algebra we use in our construction can be traced back to the work of Eilenberg and Maclane \cite{MR19092, MR20996, MR33287}. These ideas were first applied to operator algebras by V. Jones \cite{MR548118, Jo1980}  and Sutherland \cite{Su1980}. Jones used this approach in the existence portion of his remarkable classification of finite group actions on the hyperfinite $\rm{II}_{1}$ factor. 

Our construction below follows Jones fairly closely, though this may not be obvious from a casual inspection. The primary difference in our constructions and proofs is that we avoid the use of Jones' groups $\Lambda(K,G)$ to build our actions, and instead directly use trivializations of cocycles to build the desired anomalous actions. Also, the part of Lemma \ref{trivialize3cocycle} below concerning trivialization on the kernel is crucial in our applications, and does not appear in \cite{Jo1980}. For the reader interested in a more in depth discussion of related ideas, we point out the following: \cite{MR1324403}, which extracts a 3-cocycle from symmetries of a twisted crossed product, essentially going in the ``other direction" from our construction; and \cite{MR1798596}, which provides an exposition of many concepts relevant to this section.

Without further ado, we record the following theorem, which is the main tool we have for constructing anomalous actions on C*-algebras. 

\medskip

\begin{thm}\label{constructingtheaction}
Suppose that we have the following data:

\begin{itemize}
\item
A group $Q$ and $[\omega]\in H^{3}(Q, \text{U}(1))$, with a normalized representative $\omega \in Z^{3}(Q, \text{U}(1))$.

\item
A group $G$ and a surjective homomorphism $\rho:G\rightarrow Q$ with kernel $K$.

\item
A normalized cochain $c\in C^{2}(G, \text{U}(1))$ such that $dc=\rho^{*}(\omega)$.

\item 
A homomorphism $\pi: G\rightarrow \text{Aut}(B)$, where $\text{Aut}(B)$ denotes the \textit{group} of *-automorphisms of the C*-algebra $B$, rather than the 2-group.

\end{itemize}

\noindent Then there exists an $\omega$-anomalous action of $Q$ on the twisted (reduced) crossed product $B\rtimes_{\pi, c} K $, where $c\in Z^{2}(K, \text{U}(1))$ is the restriction of $c$ to $K$.

\end{thm}

The first goal of this section is to prove the above theorem. We do this in a series of lemmas. First we unpack some of the structure we are given. The equation $dc=\rho^{*}(\omega)$ unpacks to give us for all $f,g,h\in G$

\begin{equation}\label{cocycleeq}
c(f,g)c(fg,h)=\rho^{*}(\omega)(f,g,h)c(g,h)c(f,gh)=\omega(\rho(f),\rho(g),\rho(h))c(g,h)c(f,gh).
\end{equation}

\noindent Furthermore, since $\omega$ was assumed to be normalized, we have 

\begin{equation}\label{kereq}\rho^{*}(\omega)(f,g,h)=1\ \text{if any}\ f,g,h\in K.
\end{equation}

\noindent Thus the restriction of $c$ to $K$ is indeed a (normalized) 2-cocycle.

Consider the (reduced) twisted crossed product C*-algebra $B\rtimes_{\pi, c} K$ (for a detailed reference, see \cite{MR1798596, MR2288954}). The twisted crossed product has the dense subalgebra

$$(B\rtimes_{\pi, c} K)^{\circ}=C_{c}(K, B), $$

\noindent where, since $K$ is discrete, $C_{c}(K, B)$ denotes the set of finitely supported functions from $K$ to $B$. A convenient way to describe the product on this algebra is to consider unitaries $u_{g}$ satisfying the algebraic relation $u_{g}u_{h}=c(g,h) u_{gh}=c(g,h)u_{gh}$ and $u_{g}a=\pi(g)(a) u_{g}$. 

Then $(B\rtimes_{\pi, c} K)^{\circ}=\sum_{g\in K} Bu_{g}$, with multiplication induced from the above relations. Note that if $B$ is non-unital, then the $u_{g}$ do not live in $B\rtimes_{\pi, c} K$, but rather its multipler algebra. Also note that if $G$ is finite, then $(B\rtimes_{\pi, c} K)^{\circ} =B\rtimes_{\pi, c} K$.

Our goal will be to use the given data to define a $\text{U}(1)$-linear monoidal functor $\text{2-Gr}(G,\text{U}(1), \omega)\rightarrow \underline{\text{Aut}}(B\rtimes_{\pi, c} K)$. To perform this construction, pick a set theoretical section of $\rho: G\rightarrow Q$, $q\mapsto \hat{q}\in G$. Then for each $q\in Q$, we define an automorphism of $B\rtimes_{\pi, c} K$ via its action on Fourier coefficients by

$$\alpha_{q}(au_{k}):= c(\hat{q}k \hat{q}^{-1},\hat{q})^{-1}c(\hat{q},k) \pi(\hat{q})(a) u_{\hat{q}k \hat{q}^{-1}},$$

\noindent and then extends linearly to obtain a map on $(B\rtimes_{\pi, c} K)^{\circ}$.

\begin{lem} $\alpha_{q}$ extends to a $*$-automorphism of $B\rtimes_{\pi, c} K$.
\end{lem}

\begin{proof}

We will show this extends to a *-automorphism of the subalgebra $\sum_{g\in K} \text B u_{g}$. Observe the projection $E: \sum_{g\in K} \text B u_{g}\rightarrow B$ which picks off the identity component is invariant under each $\alpha_{q}$, and thus if $\alpha_{q}$ is a $*$-automorphism of $(B\rtimes_{\pi, c} K)^{\circ}$, then it will extend to a *-automorphism on the reduced completion. 

For any $k,l\in K$ and $g\in G$, using Equations \ref{cocycleeq} and \ref{kereq}, we have the following equation:

\begin{align*}
&c(g, k) c(gkg^{-1},g)^{-1}c(g,l)c(glg^{-1},g)^{-1}  \\
=&c(g, k) c(gkg^{-1},g)^{-1}\left[c(g,l)c(gkg^{-1},gl)\right]c(gkg^{-1},gl)^{-1}c(glg^{-1},g)^{-1}  \\
=&c(g, k) c(gkg^{-1},g)^{-1}\left[ \rho^{*}(\omega)(gkg^{-1},g,l)^{-1}c(gkg^{-1},g) c(gk,l)\right] c(gkg^{-1},gl)^{-1}c(glg^{-1},g)^{-1}\\
=&c(g, k) c(gk,l)c(gkg^{-1},gl)^{-1}c(glg^{-1},g)^{-1}\\
=&\left[ \rho^{*}(\omega)(g,k,l)c(k,l)c(g,kl)\right] \left[\rho^{*}(gkg^{-1},glg^{-1},g) c(gklg^{-1},g)^{-1}c(gkg^{-1},glg^{-1})^{-1}\right]\\
=&\left[c(g,kl) c(k,l)\right] \left[ c(gklg^{-1},g)^{-1}c(gkg^{-1},glg^{-1})^{-1}\right].\\
\end{align*}

\medskip

\noindent Now we compute

\begin{align*}
\alpha_{q}(au_{k})\alpha_{q}(bu_{l})&=\left[c(\hat{q}k \hat{q}^{-1},\hat{q})^{-1}c(\hat{q},k) \pi(\hat{q})(a) u_{\hat{q}k \hat{q}^{-1}}\right] \left[c(\hat{q}l \hat{q}^{-1},\hat{q})^{-1}c(\hat{q},l) \pi(\hat{q})(b) u_{\hat{q}l \hat{q}^{-1}}\right]\\
&=\left[c(\hat{q},k)c(\hat{q}k\hat{q}^{-1}, \hat{q})^{-1}c(\hat{q},l)c(\hat{q}l \hat{q}^{-1},\hat{q})^{-1}\right]c(\hat{q}k\hat{q}^{-1},\hat{q}l\hat{q}^{-1}) \pi(\hat{q})(a\pi(k)(b)) u_{\hat{q}kl\hat{q}^{-1}}.
\end{align*}

\noindent But using the first derived equation (and substituting $\hat{q}$ for $g$), we can replace the term in the square brackets to obtain

\begin{align*}
\alpha_{q}(au_{k})\alpha_{q}(bu_{l})&=c(\hat{q},kl)c(k,l)c(\hat{q}kl\hat{q}^{-1},\hat{q})^{-1} \pi(\hat{q})(a\pi(k)(b)) u_{\hat{q}kl\hat{q}^{-1}}\\
&=\alpha_{q}(c(k,l) a\pi(k)(b) u_{kl})\\
&=\alpha_{q}(au_{k}bu_{l}).\\
\end{align*}

\noindent This shows that $\alpha_{q}$ is an automorphism. To check the $*$-property, note that since $c$ is normalized we have $u^{*}_{k}=c(k,k^{-1})^{-1}u_{k^{-1}}$, and $c(k,k^{-1})=c(k^{-1},k)$ since $c|_{K}$ is a normalized 2-cocycle. Thus 

\begin{align*}
\alpha_{q}((au_{k})^{*})&=c(k,k^{-1})^{-1}\alpha_{q}(\pi(k^{-1})(a^{*})u_{k^{-1}})\\
&= c(k,k^{-1})^{-1} c(\hat{q},k^{-1}) c(\hat{q}k^{-1}\hat{q}^{-1}, \hat{q})^{-1}\pi(\hat{q}k^{-1})(a^{*})u_{\hat{q}k^{-1}\hat{q}^{-1}}.\\
\end{align*}

\noindent On the other hand, $$\alpha_{q}(au_{k})^{*} = c(\hat{q}k \hat{q}^{-1},\hat{q}k^{-1} \hat{q}^{-1})^{-1} c(\hat{q}k \hat{q}^{-1},\hat{q})c(\hat{q},k)^{-1} \pi(\hat{q}k^{-1})(a^{*}) u_{\hat{q}k^{-1} \hat{q}^{-1}}.$$

\noindent Examining the scalar in our expression for $\alpha_{q}((au_{k})^{*})$ (and again using Equations \ref{cocycleeq} and \ref{kereq}), we have

\begin{align*}
c(k,k^{-1})^{-1} c(\hat{q},k^{-1}) c(\hat{q}k^{-1}\hat{q}^{-1}, \hat{q})^{-1}
&=c(k^{-1}, k)^{-1} c(\hat{q},k^{-1})\left[c(\hat{q}k^{-1}, k)c(\hat{q}k^{-1}, k)^{-1}\right] c(\hat{q}k^{-1}\hat{q}^{-1}, \hat{q})^{-1}\\
&=c(\hat{q}k^{-1}, k)^{-1} c(\hat{q}k^{-1}\hat{q}^{-1}, \hat{q})^{-1}\\
&=c(\hat{q}k^{-1}\hat{q}^{-1}, \hat{q}k)^{-1}c(\hat{q},k)^{-1}.\\
\end{align*}

\noindent Performing a similar computation, the scalar in our above expression for $\alpha_{q}(au_{k})^{*}$ we obtain

\begin{align*}
&c(\hat{q}k^{-1}\hat{q}^{-1},\hat{q}k^{-1}\hat{q}^{-1})^{-1} c(\hat{q}k \hat{q}^{-1},\hat{q})c(\hat{q},k)^{-1}\\
=&c(\hat{q}k \hat{q}^{-1},\hat{q}k^{-1} \hat{q}^{-1})^{-1} c(\hat{q}k \hat{q}^{-1},\hat{q})\left[c(\hat{q}k^{-1}\hat{q}^{-1}, \hat{q}k) c(\hat{q}k^{-1}\hat{q}^{-1}, \hat{q}k)^{-1}\right]c(\hat{q},k)^{-1}\\
=&c(\hat{q}k^{-1}\hat{q}^{-1}, \hat{q}k)^{-1}c(\hat{q},k)^{-1}.\\
\end{align*}

\end{proof}

Now that we have built *-automophisms $\alpha_{q}$, we need intertwiners $m_{q,r}\in B\rtimes_{\pi, \delta} K$ that satisfy $m_{q,r}\alpha_{q}(\alpha_{r}(x))=\alpha_{qr}(x)m_{q,r}$ for all $x\in B\rtimes_{\pi, c} K$. Note that $\hat{q}\hat{r}=\gamma(q,r)\widehat{qr}$ for a uniquely determined (by the lift) function $\gamma: Q\times Q\rightarrow K$. 

\begin{lem}
Define the unitaries $m_{q,r}:=c(\hat{q},\hat{r})^{-1}c(\gamma(q,r),\widehat{qr}) u^{*}_{\gamma(q,r)}\in B\rtimes_{\pi, c} K$. Then for all $x\in B\rtimes_{\pi, c} K$,

$$m_{q,r}\ \alpha_{q}(\alpha_{r}(x))=\alpha_{qr}(x)\ m_{q,r}.$$

\end{lem}

\begin{proof}

Set
$$\widetilde{m}_{q,r}:=c(\hat{q},\hat{r})c(\gamma(q,r),\widehat{qr})^{-1}u_{\gamma(q,r)^{-1}}.$$

Note that $\widetilde{m}_{q,r}$ differs from $m_{q,r}$ by a unitary scalar (depending only on $q$ and $r$), and so the desired equation holds if and only if $\widetilde{m}_{q,r}\ \alpha_{q}(\alpha_{r}(x))=\alpha_{qr}(x)\ \widetilde{m}_{q,r}$. We compute the left hand side of this equation applied to a Fourier term $au_{k}$. First note that

\begin{align*}
\alpha_{q}(\alpha_{r}(au_{k}))&=c(\hat{r},k)c(\hat{r}k\hat{r}^{-1},\hat{r})^{-1}c(\hat{q},\hat{r}k\hat{r}^{-1})c(\hat{q}\hat{r}k\hat{r}^{-1}\hat{q}^{-1}, \hat{q})^{-1} c(\gamma(q,r),\widehat{qr}k\hat{r}^{-1}\hat{q}^{-1})^{-1}\\
&\times\ c(\widehat{qr}k\widehat{qr}^{-1}, \gamma(q,r)^{-1})^{-1}\ \pi(\gamma(q,r)\widehat{qr})(a)\ u_{\gamma(q,r)}\ u_{\widehat{qr}k\widehat{qr}^{-1}} u_{\gamma(q,r)^{-1}}.
\end{align*}

\noindent Then we see

$$\widetilde{m}_{q,r}\alpha_{q}(\alpha_{r}(au_{k}))=\lambda \pi(\widehat{qr})(a) u_{\widehat{qr}k\widehat{qr}^{-1}} u_{\gamma(q,r)^{-1}},$$

\noindent where

\begin{align*}
 \lambda&=c(\gamma(q,r)^{-1},\gamma(q,r))\ c(\hat{q},\hat{r})\ c(\gamma(q,r),\widehat{qr})^{-1}\ c(\hat{r},k)\ c(\hat{r}k\hat{r}^{-1},\hat{r})^{-1}\ c(\hat{q},\hat{r}k\hat{r}^{-1})\\
 &\times c(\hat{q}\hat{r}k\hat{r}^{-1}\hat{q}^{-1}, \hat{q})^{-1} c(\gamma(q,r),\widehat{qr}k\hat{r}^{-1}\hat{q}^{-1})^{-1}c(\widehat{qr}k\widehat{qr}^{-1}, \gamma(q,r)^{-1})^{-1}.
\end{align*}

\noindent But the product of the first in third term of the above expression may be rewritten

\begin{align*}
&c(\gamma(q,r)^{-1},\gamma(q,r))c(\gamma(q,r), \widehat{qr})^{-1}\\
=&\left[c(\gamma(q,r)^{-1}, \hat{q}\hat{r})c(\gamma(q,r)^{-1}, \hat{q}\hat{r})^{-1}\right]c(\gamma(q,r), \widehat{qr})^{-1}c(\gamma(q,r)^{-1},\gamma(q,r))\\   
=&c(\gamma(q,r)^{-1},\hat{q}\hat{r}).\\
\end{align*}

\noindent After making this substitution and rearranging, we have 

\begin{align*}
\lambda&=\left[  \ c(\hat{r},k)\ c(\hat{r}k\hat{r}^{-1},\hat{r})^{-1}\ c(\hat{q},\hat{r}k\hat{r}^{-1})\right] \left[ c(\hat{q}\hat{r}k\hat{r}^{-1}\hat{q}^{-1}, \hat{q})^{-1} c(\hat{q},\hat{r})\right]\\
&\times \left[ c(\gamma(q,r),\widehat{qr}k\hat{r}^{-1}\hat{q}^{-1})^{-1}c(\widehat{qr}k\widehat{qr}^{-1}, \gamma(q,r)^{-1})^{-1} c(\gamma(q,r)^{-1},\hat{q}\hat{r})\right].
\end{align*}

\noindent Using Equations \ref{cocycleeq} and \ref{kereq} we see the term in the second bracket can be rewritten as

$$ c(\hat{q}\hat{r}k\hat{r}^{-1}\hat{q}^{-1}, \hat{q})^{-1} c(\hat{q},\hat{r}) =c(\hat{q}\hat{r}k\hat{r}^{-1}, \hat{r})c(\hat{q}\hat{r}k\hat{r}^{-1}\hat{q}^{-1}, \hat{q}\hat{r})^{-1}.$$

\noindent Then

\begin{align*}
\lambda&=\left[  \ c(\hat{r},k)\ c(\hat{r}k\hat{r}^{-1},\hat{r})^{-1}\ c(\hat{q},\hat{r}k\hat{r}^{-1}) c(\hat{q}\hat{r}k\hat{r}^{-1}, \hat{r}) \right]\\
&\times \left[ c(\hat{q}\hat{r}k\hat{r}^{-1}\hat{q}^{-1}, \hat{q}\hat{r})^{-1} c(\gamma(q,r),\widehat{qr}k\hat{r}^{-1}\hat{q}^{-1})^{-1}c(\widehat{qr}k\widehat{qr}^{-1}, \gamma(q,r)^{-1})^{-1} c(\gamma(q,r)^{-1},\hat{q}\hat{r})\right].
\end{align*}

\noindent But the first bracket in the above expression can be rewritten (using Equations \ref{cocycleeq} and \ref{kereq}) as

$$\ c(\hat{r},k)\ c(\hat{r}k\hat{r}^{-1},\hat{r})^{-1}\ c(\hat{q},\hat{r}k\hat{r}^{-1}) c(\hat{q}\hat{r}k\hat{r}^{-1}, \hat{r})=c(\hat{r},k)c(\hat{q},\hat{r}k)=c(\hat{q},\hat{r})c(\hat{q}\hat{r},k),$$

\noindent while for the second we have

\begin{align*}
&c(\hat{q}\hat{r}k\hat{r}^{-1}\hat{q}^{-1}, \hat{q}\hat{r})^{-1} c(\gamma(q,r),\widehat{qr}k\hat{r}^{-1}\hat{q}^{-1})^{-1}c(\widehat{qr}k\widehat{qr}^{-1}, \gamma(q,r)^{-1})^{-1} c(\gamma(q,r)^{-1},\hat{q}\hat{r}) \\
=&c(\gamma(q,r),\widehat{qr}k)^{-1}c(\widehat{qr}k\widehat{qr}^{-1}, \widehat{qr})^{-1}.\\
\end{align*}

\noindent Altogether, this implies

\begin{align*}
\lambda&=c(\hat{q},\hat{r})\left[c(\hat{q}\hat{r},k)c(\gamma(q,r),\widehat{qr}k)^{-1}\right]c(\widehat{qr}k\widehat{qr}^{-1}, \widehat{qr})^{-1} \\
&=c(\hat{q},\hat{r})c(\gamma(q,r),\widehat{qr})^{-1}c(\widehat{qr},k)c(\widehat{qr}k\widehat{qr}^{-1}, \widehat{qr})^{-1}.
\end{align*}

\noindent On the other hand, 

$$\alpha_{qr}(au_{k})\ \widetilde{m}_{q,r}=c(\hat{q},\hat{r})c(\gamma(q,r),\widehat{qr})^{-1}c(\widehat{qr}k\widehat{qr}^{-1},\widehat{qr})^{-1}c(\widehat{qr},k)\pi(\widehat{qr})(a)u_{\widehat{qr}k\widehat{qr}^{-1}}u_{\gamma(q,r)^{-1}}.$$

\end{proof}

The following lemma concludes the proof of Theorem \ref{constructingtheaction}:

\begin{lem}\label{action}
With notation as above, we have $$ \omega(q,r,s)m_{qr,s}m_{q,r}= m_{q,rs}\alpha_{q}(m_{r,s}).$$ In particular the assignment $q\mapsto \alpha_{q}$ extends to a $\text{U}(1)$-linear functor $(G, \text{U}(1), \omega)\rightarrow \underline{\text{Aut}}(B \rtimes_{\pi, c} K)$.

\end{lem}

\begin{proof}

We compute $$m_{q,rs}\alpha_{q}(m_{r,s})= \lambda u^{*}_{\hat{q}\gamma(r,s)\hat{q}^{-1}\gamma(q,rs)},$$

\noindent where 
\begin{align*}
 \lambda&=\left[c(\hat{r},\hat{s})^{-1}c(\gamma(r,s),\widehat{rs})c(\hat{q},\gamma(r,s))^{-1}\right]\\
 &\times \left[c(\hat{q}\gamma(r,s)\hat{q}^{-1},\hat{q})c(\hat{q},\widehat{rs})^{-1}c(\gamma(q,rs),\widehat{qrs})c(\hat{q}\gamma(r,s)\hat{q}^{-1},\gamma(q,rs))^{-1}\right].\\
\end{align*}

\noindent Then we have

\begin{align*}
\lambda&=\left[c(\hat{r},\hat{s})^{-1}c(\gamma(r,s),\widehat{rs})c(\hat{q},\gamma(r,s))^{-1}c(\hat{q}\gamma(r,s),\widehat{rs})^{-1}\right]\\
&\times \left[c(\hat{q}\gamma(r,s),\widehat{rs}) c(\hat{q}\gamma(r,s)\hat{q}^{-1},\hat{q})c(\hat{q},\widehat{rs})^{-1}c(\gamma(q,rs),\widehat{qrs})c(\hat{q}\gamma(r,s)\hat{q}^{-1},\gamma(q,rs))^{-1}\right]\\
&=\left[c(\hat{r},\hat{s})^{-1}c(\hat{q},\hat{r}\hat{s})^{-1}\right]\times \left[c(\hat{q}\gamma(r,s)\hat{q}^{-1}, \hat{q}\widehat{rs})c(\gamma(q,rs), \widehat{qrs})c(\hat{q}\gamma(r,s)\hat{q}^{-1},\gamma(q,rs))^{-1}\right]\\
&=c(\hat{r},\hat{s})^{-1}c(\hat{q},\hat{r}\hat{s})^{-1} c(\gamma(q,r)\gamma(qr,s),\widehat{qrs}).
\end{align*}

\noindent where in the last step we use the equality $\hat{q}\gamma(r,s)\hat{q}^{-1}\gamma(q,rs)=\gamma(q,r)\gamma(qr,s)$ obtained by expanding $\hat{q}\hat{r}\hat{s}$ two ways. On the other hand, we have

$$m_{qr,s}m_{q,r}=\lambda^{\prime}u^{*}_{\gamma(q,r)\gamma(qr,s)}=\lambda^{\prime}u^{*}_{\hat{q}\gamma(r,s)\hat{q}^{-1}\gamma(q,rs)},$$

where 

\begin{align*}
\lambda^{\prime}&=c(\hat{q},\hat{r})^{-1}c(\gamma(q,r),\widehat{qr})c(\widehat{qr},\widehat{qr}\hat{s})^{-1}c(\gamma(qr,s),\widehat{qrs})c(\gamma(q,r),\gamma(qr,s))^{-1}\\
&=\left[c(\hat{q},\hat{r})^{-1}c(\hat{q}\hat{r},\hat{q}\hat{r}\hat{s})^{-1}\right]  \times \left[ c(\hat{q}\hat{r},\hat{q}\hat{r}\hat{s})c(\gamma(q,r),\widehat{qr})c(\widehat{qr},\widehat{qr}\hat{s})^{-1}c(\gamma(qr,s),\widehat{qrs})c(\gamma(q,r),\gamma(qr,s))^{-1}\right]\\
&=\left[\rho^{*}(\omega)(\hat{q},\hat{r},\hat{s})^{-1}c(\hat{r},\hat{s})^{-1}c(\hat{q},\hat{r}\hat{s})^{-1}\right]\times \left[c(\gamma(q,r),\widehat{qr}\hat{s})c(\gamma(qr,s),\widehat{qrs})c(\gamma(q,r),\gamma(qr,s))^{-1}\right]\\
&=\omega^{-1}(q,r,s)\lambda.\\
\end{align*}

\end{proof}

\subsection{Free group examples.}\label{freegroupex} An easy class of examples to which we can easily apply Theorem \ref{buildingactions} involves free groups. Let $\mathbbm{F}_{n}$ denote the free group with $n$ generators. These have $H^{k}(\mathbbm{F}_{n}, \text{U}(1))=0$ for all $k\ge 2$. For any finitely generated $Q$ and $\omega\in H^{3}(G, \text{U}(1))$, we can find a surjection $\rho:F_{n}\rightarrow G$, which is completely determined by picking $n$ elements in $Q$ which generate. The kernel of $\rho$ is a free normal subgroup $K\le F_{n}$, with $K\cong \mathbbm{F}_{m}$ (with $m$ possibly $\infty$).

Restricting to the case $|Q|<\infty$ and $n>1$, we have the formula (e.g. \cite[Proposition 3.9]{MR1812024})

$$m=|Q|(n-1)+1.$$

\noindent Since $H^{3}(\mathbbm{F}_{n}, \text{U}(1))=1=H^{2}(\mathbbm{F}_{r},\text{U}(1))$, we can find a trivialization $c\in C^{2}(\mathbbm{F}_{n}, \text{U}(1))$ of $\rho^{*}(\omega)$ such that $c|_{K\cong \mathbbm{F}_{r}}=1$. Applying Theorem \ref{constructingtheaction} to the case $B=\mathbbm{C}$ immediately gives us the following Corollary:

\begin{cor}\label{freegroupcor}
Let $Q$ be a finite group, $\omega\in Z^{3}(Q, \text{U}(1))$, and pick a surjective homomorphism $\mathbbm{F}_{n}\rightarrow Q$, with $n>1$. Then there exists an $\omega$-anomalous $Q$ action on the reduced group C*-algebra $C^{*}_{r}(\mathbbm{F}_{m})$, where $m=|Q|(n-1)+1$.
\end{cor}

If we consider the case $n=1$, then to have a surjective homomorphism $\rho:\mathbbm{F}_{1}\cong \mathbbm{Z}\rightarrow Q$, we must have $Q=\mathbbm{Z}/m\mathbbm{Z}$. In the case, the kernel of $\rho$ is $m\mathbbm{Z}\cong \mathbbm{Z}$. Recall that for an irrational real number, we can consider the action of $\mathbbm{Z}$ on $C(S^{1})$ via $g(f(z))=f(e^{-2\pi i\theta}z)$, where $g$ is the cyclic generator of $\mathbbm{Z}$. Then the \textit{irrational rotation algebra} $A_{\theta}$ is defined as the crossed product $A:=C(S^{1})\rtimes \mathbbm{Z}$. We have the following Corollary:

\begin{cor}\label{Irrationalrotation}
For any irrational real number $\theta$ and any $\omega\in Z^{3}(\mathbbm{Z}/m\mathbbm{Z}, \text{U}(1))$, there exists an an $\omega$-anomalous action of $\mathbbm{Z}/m\mathbbm{Z}$ on the irrational rotation algebra $A_{\theta}$.
\end{cor}

\begin{proof}
Consider the $\mathbbm{Z}$ action of on $C(S^{1})$ which rotates by the angle $e^{2 \pi i \frac{\theta}{m}}$. Then we can apply Theorem \ref{constructingtheaction} to $m\mathbbm{Z}$ to obtain an $\omega$-anomalous action on $C(S^{1})\rtimes m\mathbbm{Z}\cong A_{\theta}$.
\end{proof}

\subsection{A cohomology lemma.}\label{cohomologylemma}

The previous section shows we can always use free groups to trivialize cohomology, which allows us to apply Theorem $\ref{constructingtheaction}$ to build anomalous actions of finite groups on C*-algebras. However, to build examples related to geometry and topology, we need a finite group $G$ to surject onto $Q$. Furthermore, it is crucial for the applications in this paper the twisted crossed product we act on is actually untwisted (this was automatic for free groups, which have trivial $H^{2}$ and $H^{3}$. This means we need to find specific trivializations of the pullback which restrict to the trivial 2-cohomology class on the kernel. To this end, we have the following lemma, which should be compared to \cite[Lemma 7.1.12]{Jo1980}:

\begin{lem}\label{trivialize3cocycle} Let $Q$ be a finite group and $[\omega_{0}]\in H^{3}(Q, \text{U}(1))$. Then there exists a finite group $G$, a surjective homomorphism $\rho: G\rightarrow Q$, a normalized unitary 2-cochain $c\in C^{2}(G, \text{U}(1))$, and a normalized unitary 3-cocycle representative $\omega\in [\omega_{0}]$ such that $dc =\rho^{*}(\omega)$ and $c|_{Ker(\rho)}=1$.

\end{lem}

\begin{proof}
We begin by noting that that our proof uses standard ``dimension shifting" arguments in cohomology, and thus closely parallels the construction in the proof of \cite[Lemma 7.1.2]{Jo1980}. Ultimately we use the same construction of $G$, but we require a more thorough examination to find an appropriate trivialization.

To begin, we need to pick an representative of $[\omega_{0}]$. First let $n=|Q|$, and consider $\mathbbm{Z}/n\mathbbm{Z}$ as a trivial $Q$ module, and let $$M=\text{Hom}_{\mathbbm{Z}}(\mathbbm{Z}[Q], \mathbbm{Z}/n\mathbbm{Z})\cong \prod_{g\in Q} \mathbbm{Z}/n\mathbbm{Z}$$ be the coinduced module, where the action on the latter permutes the factors via left multiplication by elements in $Q$. Then the diagonal inclusion $z\rightarrow (z, \dots, z)$ is an inclusion of $Q$ modules, and thus we have a short exact sequence of $Q$-modules

$$0\rightarrow \mathbbm{Z}/n\mathbbm{Z}\rightarrow M\rightarrow A \rightarrow 0,$$

\noindent with $A:= M / (\mathbbm{Z}/n\mathbbm{Z})$. Then since $M$ is cohomologicaly trivial by Shapiro's Lemma \cite[Chapter III, Proposition 6.2]{MR672956}, the connecting homomorphism $\delta: H^{2}(Q, A)\rightarrow H^{3}(Q, \mathbbm{Z}/n\mathbbm{Z})$ is an isomorphism. In this argument, we require a more detailed description of the connecting homomorphism at the level of coboundaries. It can be defined as follows:

Pick a set-theoretical (unital) lift $A \rightarrow M$, $a\mapsto \hat{a}$. Then for any 2-cocyle $\alpha\in Z^{2}(Q, A)$, we obtain a 2-cochain $\hat{\alpha}\in C^{2}(Q, M)$. Taking the coboundary map $d(\hat{\alpha})$, we see that the values of the function in fact live in the subgroup $\mathbbm{Z}/n\mathbbm{Z}\le M$, since applying the quotient map $M\rightarrow A$ to $\hat{\alpha}$ is a 2-cocycle by construction. Thus we can define the connecting map at the level of cocycles by $\delta(\alpha):=d(\hat{\alpha})\in Z^{3}(Q, \mathbbm{Z}/n\mathbbm{Z})$. This is a 3-coboundary as an M-valued cochain but \textit{not} as a $\mathbbm{Z}/n\mathbbm{Z}$-valued cochain. In fact, since this is induces isomorphism in cohomology, every $\mathbbm{Z}/n\mathbbm{Z}$ valued cochain is cohomologous to one of this form. 

It is well known that the natural inclusion $i:\mathbbm{Z}/n\mathbbm{Z}\hookrightarrow \text{U}(1)$ which embeds the cyclic group as roots of unity induces a surjection $i_{*}: H^{3}(Q, \mathbbm{Z}/n\mathbbm{Z})\rightarrow H^{3}(Q, \text{U}(1))$. This follows directly from the long exact sequence induced from the short exact sequence $$0\rightarrow \mathbbm{Z}/n\mathbbm{Z}\rightarrow \text{U}(1)\rightarrow \text{U}(1)\rightarrow 0,$$

\noindent as the induced map $H^{3}(Q, \text{U}(1))\rightarrow H^{3}(Q, \text{U}(1))$ raises a $\text{U}(1)$ valued function to the $n^{th}$ power. This is the zero map on cohomology \cite[Chapter III, Proposition 10.1.]{MR672956}. Thus $H^{3}(Q, \mathbbm{Z}/n\mathbbm{Z})\rightarrow H^{3}(Q, \text{U}(1))$ is surjective. Therefore we can pick a representative of our cohomology class $\omega=i_{*}(\delta(\alpha))$ for some (normalized) cocycle $\alpha\in Z^{2}(Q, A)$.

Define the group $G:=A\rtimes_{\alpha}Q$ as the twisted crossed product of $Q$ with $A$. Explicitly, elements are pairs $(a,q)$ with product

$$(a,q)\cdot (a^{\prime},q^{\prime})=(a+q(a^{\prime})+\alpha(q,q^{\prime}), qq^{\prime}).$$

\noindent Let $\rho:G\rightarrow Q$ be the surjective homomorphism which projects onto the $Q$ factor.

Since $M$ is an abelian group there exists a normalized 2-cocycle $\beta\in H^{2}(A, \mathbbm{Z}/n\mathbbm{Z})$ (where the coefficients have a trivial $A$ modules structure) such that 

$$M\cong \mathbbm{Z}/n\mathbbm{Z}\times_{\beta} A,$$

\noindent where the latter group is defined similarly to the case above, with group operation defined on pairs $(z,a)+(z^{\prime}, a^{\prime}):=(z+z^{\prime}+ \beta(a,a^{\prime}), a+a^{\prime})$. Since $M$ is abelian, $\beta$ is symmetric, i.e. $\beta(a,a^{\prime})=\beta(a^{\prime},a)$. Using such an isomorphism, we can thus choose the lift $A\mapsto M$ via $\hat{a}=(0,a)$.

Then the projection onto $q$ gives a short exact sequence $$0\rightarrow A\rightarrow G\rightarrow Q\rightarrow 0.$$

Our goal is to find an explicit trivialization of $\rho^{*}(\omega)\in Z^{2}(G, \mathbbm{Z}/n\mathbbm{Z})$, where $\rho^{*}$ denotes the pullback under $\rho$ of cochains (i.e. the inflation homomorphism). We will then compute the corresponding 2-cocycle obtained upon restriction to $A\le G$.

Note that the function $c(a,q)=a$ in $C^{1}(G, A)$ satisfies 

$$dc((a, q), (a^{\prime},q^{\prime}))=-q(a^{\prime})+(a+q(a^{\prime})+\alpha(q,q^{\prime}))-a=\alpha(q,q^{\prime})=\rho^{*}(\alpha)((a,q),(a^{\prime},q^{\prime})).$$

Thus $\rho^{*}(\alpha)\in Z^{2}(G, A)$ is a coboundary. Naturality of the connecting homomorphisms on cohomology immediately gives us that $\rho^{*}(\omega)$ is trivial, but that is not sufficient. We need to examine a specific trivialization.

Consider the lifted function $\hat{c}(a,q)=(0,a)\in M=\mathbbm{Z}/n\mathbbm{Z}\times_{\beta} A$. Then $d\hat{c}=(t, \rho^{*}(\alpha))$ for some function $t:G\times G\rightarrow \mathbbm{Z}/n\mathbbm{Z}$. Therefore, if we can consider the function 

$$c_{0}:=-d\hat{c}+\widehat{\rho^{*}(\alpha)}\in C^{2}(G, \mathbbm{Z}/n\mathbbm{Z}),$$

we see

$$dc_{0}=d(-d\hat{c})+d(\widehat{\rho^{*}(\alpha)})=\delta(\rho^{*}(\alpha))=\rho^{*}(\delta(\alpha))=\rho^{*}(\omega),$$

\noindent where the second to last equality follows from commutativity of the diagram

\[\begin{tikzcd}
    Z^{2}(G, A) \arrow{r}{\delta} & Z^{3}(G, \mathbbm{Z}/n\mathbbm{Z}) \\
    Z^{2}(Q, A) \arrow{r}{\delta} \arrow{u}{\rho^{*}} & Z^{3}(Q, \mathbbm{Z}/n\mathbbm{Z})  \arrow{u}{\rho^{*}}
 \end{tikzcd}\]

\noindent with $\delta$ defined explicitly as above with respect to our fixed uniform lift $A\rightarrow M$, $a\mapsto (0,a)$. Thus $dc_{0}=\rho^{*}(\omega)$ in $Z^{3}(G, \mathbbm{Z}/n\mathbbm{Z})$, and thus embedding everything into roots of unity we see that $c_{0}$ also gives an explicit trivialization of $\rho^{*}(\omega)$ in $Z^{3}(Q, \text{U}(1))$.

All that remains is to compute $c_{0}|_{A}$, which is a $\mathbbm{Z}/n\mathbb{Z}$-valued 2-cocycle since $\rho^{*}(\omega)|_{A}=0$ (since $\omega$ was assumed to be normalized and $A$ is the kernel of the projection $G\rightarrow Q$).

We see 

\begin{align*}
c_{0}((a,0), (a^{\prime},0))&= -\left( -(0,a^{\prime})+(0, a+a^{\prime})-(0, a) \right)\\
&= (\beta(a^{\prime},a),a^{\prime}+a)+(-\beta(a+a^{\prime},-a-a^{\prime}), -a-a^{\prime})\\
&=(\beta(a^{\prime},a), 0).\\
\end{align*}

Thus $c_{0}|_{A}=\beta \in Z^{2}(A, \mathbbm{Z}/n\mathbb{Z})$. This cocycle is non-trivial in general (since the extension $M$ of $A$ by $\mathbbm{Z}/n\mathbbm{Z}$ is not split in general). However, since it is a symmetric cocycle, the inclusion of $c_{0}$ into $Z^{2}(A, \text{U}(1))$ is cohomologically trivial. This is most easily seen by noting that interpreting $c_{0}\in Z^{2}(A, \text{U}(1))$ gives an abelian extension, $$0\rightarrow \text{U}(1)\rightarrow M^{\prime}\rightarrow A,$$

\noindent which corresponds to an element in $\text{Ext}^{1}_{\mathbbm{Z}}(A, \text{U}(1))$, and this extension splits if and only if and only if the cocycle is trivial. But since $\text{U}(1)$ is divisible (as an abelian group), it is injective as a $\mathbbm{Z}$-module, and thus $\text{Ext}^{1}_{\mathbbm{Z}}(A, \text{U}(1))=0$. Therefore the above extension is split, which implies $c_{0}|_{A}\in Z^{2}(A, \text{U}(1))$ is a coboundary. In particular, there exists a function $f\in C^{1}(A, \text{U}(1))$ such that $df=c_{0}|_{A}$ (note that $f$ \textit{cannot} in general be chosen to have coefficients in $\mathbbm{Z}/n\mathbbm{Z}$!). Let $\tilde{f}$ be an arbitrary extension of $f$ from $A$ to all of $G$, so that $f\in C^{1}(G, \text{U}(1))$. 

Then the cochain $c:= c_{0} d\tilde{f}^{-1}\in C^{2}(G, \text{U}(1))$ (we have switched to multiplicative notation since our coefficients are $\text{U}(1)$) satisfies $dc=dc_{0}=\rho^{*}(\omega)$, and $c|_{A}=1$ as desired.

\end{proof}

We note that a version of the above lemma (that does not require the additional trivialization on the kernel) is also used in \cite{EG2018} to realize all twisted doubles as representation categories of a completely rational conformal field theory.

What can we say about the groups $G$ we construct in the above theorem? If $Q$ has order $n$, and we consider the trivial anomaly in $H^{3}$, then the above construction outputs $G=(\mathbbm{Z}/n\mathbbm{Z}\wr Q)/\mathbb{Z}/n\mathbbm{Z}$, where $\mathbb{Z}/n\mathbbm{Z}$ embeds as the diagonal normal subgroup. Given $\omega\in Z^{3}(Q, \text{U}(1))$, the construction above demonstrates how to find an appropriate element $ \alpha\in Z^{2} ( Q, (\prod_{Q} \mathbbm{Z}/n\mathbbm{Z})/\mathbbm{Z}/n\mathbbm{Z}\ )$, so that $G$ is then a ``twist by $\alpha$ of $( \mathbbm{Z}/n\mathbbm{Z}\wr Q)/(\mathbb{Z}/n\mathbbm{Z})$. In general this group will be quite large.

In the case of $Q=\mathbbm{Z}/2\mathbbm{Z}$ we can be very explicit about $G$. Note that $\mathbbm{Z}/2\mathbbm{Z}\times \mathbbm{Z}/2\mathbbm{Z}/\mathbbm{Z}/2\mathbbm{Z}\cong \mathbbm{Z}/2\mathbbm{Z}$, and the $\mathbbm{Z}/2\mathbbm{Z}$ action is trivial, thus $G=\mathbbm{Z}/2\mathbbm{Z}\times_{\alpha} \mathbbm{Z}/2\mathbbm{Z}$ for $\alpha\in H^{2}(\mathbbm{Z}/2\mathbbm{Z},\mathbbm{Z}/2\mathbbm{Z})\cong \mathbbm{Z}/2\mathbbm{Z}$. Thus for the non-trivial $\omega\in H^{3}(\mathbbm{Z}/2\mathbbm{Z}, \text{U}(1))$, we see that our construction produces the group $\mathbbm{Z}/4\mathbbm{Z}$, which is an extension of $\mathbbm{Z}/2\mathbbm{Z}$ by $\mathbbm{Z}/2\mathbbm{Z}$

$$0\rightarrow \mathbbm{Z}/2\mathbbm{Z}\rightarrow \mathbbm{Z}/4\mathbbm{Z}\rightarrow \mathbbm{Z}/2\mathbbm{Z}\rightarrow 0$$

\noindent In particular, the canonical quotient $\rho: \mathbbm{Z}/4\mathbbm{Z}\rightarrow \mathbbm{Z}/2\mathbbm{Z}$ satisfies $\rho^{*}(\omega)$ is trivial, and further we can choose a trivialization that restricts trivially to $\text{Ker}(\rho)=\mathbbm{Z}/2\mathbbm{Z}$. Thus to build an anomalous $\mathbbm{Z}/2\mathbbm{Z}$ action on a C*-algebra, we can look for $\mathbbm{Z}/4\mathbbm{Z}$ actions and apply the machinery developed above.

\section{Anomalies in topology.}\label{anomaliesintopology}

We will find some anomalous actions of finite group on (algebras Morita equivalent to) commutative C*-algebra on compact connected Hausdorff spaces. As mentioned in the introduction, finding anomalous actions on disconnected spaces in general is much easier, e.g. on finite sets. Let $\mathcal{K}$ denote the C*-algebra of compact operators on a separable infinite dimensional Hilbert space. We will use the results from above to build anomalous actions on stabilizations of commutative C*-algebras $C(X)\otimes \mathcal{K}$ for compact connected Hausdorff spaces $X$. In fact, we will build actions on C*-algebras Morita equivalent to $C(X)$, realized as the endomorphism algebras of finite dimensional locally trivial Hilbert modules over $X$. However, as these are all stably isomorphic to $C(X)$, we find it more convenient to use the language of stabilized C*-algebras.

\begin{thm}\label{classofspacesthm}
Let $\mathcal{C}$ be a class of compact, connected Hausdorff spaces such that every finite group admits a free action on some space $X\in \mathcal{C}$. Let $\mathcal{C}/\text{Ab}$ be the class of spaces obtained as quotients of members of $X$ by free actions of abelian groups. Then for any finite group $Q$ and any anomaly $\omega\in Z^{3}(Q, \text{U}(1))$, there exists a space $X\in \mathcal{C}/\text{Ab}$ and an $\omega$-anomalous action of $Q$ on $C(X)\otimes \mathcal{K}$
\end{thm}

\begin{proof}
Let $Q$ be a group, $\rho:G\rightarrow Q$ a surjective homomorphism with abelian kernel $K$, and $c\in C^{2}(Q, \text{U}(1))$ a (normalized) $\text{U}(1)$-valued 2-cochain satisfying $dc=\rho^{*}(\omega)$ and $c|_{Ker(\rho)}=1$. These exist by Lemma \ref{trivialize3cocycle}.

Let $X\in \mathcal{C}$ be a space with a free $G$ action, and consider the crossed product $ C(X)\rtimes K$. Then by Theorem \ref{constructingtheaction} we obtain an $\omega$-anomalous action of $G$ on $C(X)\rtimes K$. But since the action by $K\le Q$ is free, by \cite[Theorem 14]{MR453917}, $C(X)\rtimes K=\text{End}(H)$ where $H\in \text{Hilb}(X/K)$ is a finite dimensional locally trivial Hilbert bundle over $X/K$ (local triviality follows from freeness of the action and the fact that $K$ is finite). Since $X/K$ is connected, every projection $p\in M_{n}(C(X/K))$ is a full projection \cite[6.3.6]{MR1656031}. Thus $\text{End}_{X/K}(H)\cong pM_{n}(C(X/K))p $ is strongly Morita equivalent to $C(X/K)$, and we have an isomorphism $C(X/K)\otimes \mathcal{K}\cong (C(X)\rtimes K)\otimes \mathcal{K}$. Furthermore, we have a natural monoidal inclusion $\underline{\text{Aut}}(C(X/K))\rightarrow \underline{\text{Aut}}(C(X/K))\otimes \mathcal{K}$, where automorphisms act only on the first tensor category and unitaries in $C(X/K)$ naturally include (diagonally) into the multiplier algebra. Thus we obtain an $\omega$-anomalous action on $C(X/K)\otimes \mathcal{K}$. By definition, $X/K\in \mathcal{C}/\text{Ab}$.

\end{proof}

We can now use this theorem to obtain large classes of spaces on which every pointed fusion category acts. We point out two classes of examples which seemed interesting to us:

\subsection{Manifolds of with dimension $n\ge 2$ (Proof of Theorem \ref{actionsonmanifolds}). }\label{allmanifolds} 

Since the class of closed, connected $n$-manifolds is closed under taking quotients by free actions of finite abelian groups, it suffices to show every finite group admits a free action on an n-manifold for $n\ge 2$. Then an application of Theorem \ref{classofspacesthm} will show that every finite group admits an $\omega$-anomalous action on a closed connected $n$-manifold for every anomaly.

We first consider the case $n=2$. It was pointed out to us by Andr\'{e} Henriques that every finite group has a free action on a closed connected surface (this is a classical result well known to topologists). Recall that the surface $T_{g}$ of genus $g$ has fundamental group $\pi_{1}(T_{g})=\langle x_{1},y_{1}\dots, x_{g},y_{g}\ |\ [x_{1},y_{1}]\cdots [x_{g},y_{g}]=1\rangle$. In particular, if a finite group $G$ has $g$ generators, we have a surjection $\rho: \pi_{1}(T_{g})\rightarrow G$ that sends $x_{i}$ to the generators and $y_{i}$ to $1$. $\text{Ker}(\rho)\le \pi_{1}(T_{g})$ defines a covering $\tilde{T}_{g}$ of $T_{g}$ with covering group $\pi_{1}(T_{g})/\text{Ker}(\rho)\cong G$. Since $G$ is finite $\tilde{T}_{g}$ is a closed connected surface with free $G$-action.

For $n\ge 3$, we can just take a free action on a surface and take the product with a $n-2$ dimensional surface. Alternatively, pick a surjection from a free group $\mathbbm{F}_{m}\rightarrow G$ with kernel $K$, and consider the m-fold connected sum $(S^{1}\times S^{n-1})\#\dots \# (S^{1}\times S^{n-1})$, whose fundamental group is $\mathbbm{F}_{m}$ via van Kampen's theorem. Then from covering space theory, the universal cover has a free action of $\mathbbm{F}_{m}$ whose quotient by $K$ is a $\mathbbm{F}_{m}/K\cong G$-covering. In particular, since $G$ is finite this is again a closed connected n-manifold with a free $G$ action.

\subsection{Manifolds with $H^{1}(X, \mathbbm{Z})=0$ and dimension $n\ge 4$. (Proof of Theorem \ref{actionsonmanifolds})}\label{manifoldspi1finite}
We use the same type of logic as the previous example to obtain $\omega$-anomalous actions on $C(M)\otimes\mathcal{K}$ with $H^{1}(M,\mathbbm{Z})=0$, when $n\ge 4$. First we claim that every finite group $G$ admits a free action on a simply connected $n$-manifold for $n\ge 4$. It is well-known that every finitely generated group can be realized as the fundamental group of some closed connected $n$-manifold for $n\ge 4$. For example\footnote{while this is elementary algebraic topology, the description here follows an answer to a MathOverflow question \cite{15414}}, given a presentation of $G$ with $m$ generators and relations $r_{1},\dots r_{k}$, as in the previous construction, we can take the m-fold connected sum $X:=(S^{1}\times S^{n-1})\# \dots \# (S^{1}\times S^{n-1})$, which has fundamental group $\mathbbm{F}_{m}$. Identifying the generators of the fundamental group with generators in our presentation, we can consider the loop in $X$ corresponding to each word $r_{i}$. We can perform surgery on the tubular neighborhood of these loops, which are homeomorphic to $S^{1}\times D^{n-1}$, by replacing them with $D^{2}\times S^{n-2}$, which are simply connected (since $n\ge 4$). This results in a compact connected manifold $Y$. By van Kampen's theorem $\pi_{1}(Y)=G$. The universal cover $\tilde{Y}$ is then a compact connected $n$-manifold, with free $G$ action such that $\tilde{Y}/G=Y$. 

Now we claim that if $X$ is simply connected and $A$ is an abelian group that acts freely on $X$, then $X/A$ has the property that $H^{1}(X/A, \mathbbm{Z})=0$. Once we have this, the result follows from Theorem \ref{classofspacesthm}. To see the claim, note by the universal coefficient theorem, we have a short exact sequence

$$0\rightarrow \text{Ext}(H_{0}(X/A, \mathbbm{Z}), \mathbbm{Z})\rightarrow H^{1}(X/A, \mathbbm{Z})\rightarrow \text{Hom}(H_{1}(X/A, \mathbbm{Z}), \mathbbm{Z})\rightarrow 0.$$

\noindent The $\text{Ext}$-term vanishes since $H_{0}(X/A,\mathbbm{Z})=\mathbbm{Z}$ is free. Furthermore, since $H_{1}(X/A, \mathbbm{Z})=\pi_{1}(X/A)=A$ (since $A$ is abelian) is finite, the last term is also $0$. Thus $H^{1}(X, \mathbbm{Z})$ is $0$.

\subsection{Anomalous actions of cyclic groups on inductive limits of $S^{1}$.}\label{ActionsonS}. We briefly recall the $n=1$ case of Corollary \ref{freegroupcor}. Let $Q=\mathbbm{Z}/n\mathbbm{Z}$ and $G=\mathbbm{Z}$. Then $H^{3}(\mathbbm{Z}/n\mathbbm{Z}, \text{U}(1))\cong \mathbbm{Z}/n\mathbbm{Z}$ and $H^{3}(\mathbbm{Z}, \text{U}(1))$ is trivial. The natural quotient homomorphism $\rho_{n}:\mathbbm{Z}\rightarrow \mathbbm{Z}/n\mathbbm{Z}$ satisfies $[\rho^{*}(\omega)]=[1]$ for all $\omega\in Z^{3}(\mathbbm{Z}/n\mathbbm{Z}, \text{U}(1))$. In addition, $H^{2}(\mathbbm{Z}, \text{U}(1))$ is also trivial so we can choose a trivialization $c\in C^{2}(\mathbbm{Z}, \text{U}(1))$ with $dc=\rho^{*}(\omega)$ and $c|_{\text{Ker}(\rho)}=1$. 

Now, consider the trivial action of $\mathbbm{Z}$ on a point. Then since $\text{Ker}(\rho)=n\mathbbm{Z}\cong \mathbbm{Z}$, we have $C^{*}_{r}(\text{Ker}(\rho))\cong C^{*}_{r}(\mathbbm{Z})\cong C(S^{1})$. Then for all 3-cocycles $\omega\in Z^{3}(\mathbbm{Z}/n\mathbbm{Z}, U(1))$, Theorem \ref{constructingtheaction} gives us an $\omega$-anomalous action of $\mathbbm{Z}/n\mathbbm{Z}$ on $C(S^{1})$.

\bigskip

To make this more interesting, we will take inductive limits of anomalous circle actions. Let $p$ be a prime with $p=1\ \text{mod}\ n$. Then consider the inductive limit of abelian groups.

\[\begin{tikzcd}
    \mathbbm{Z} \arrow{r}{p} & \mathbbm{Z} \arrow{r}{p} & \mathbbm{Z} \arrow{r}{p} & \dots
 \end{tikzcd}\]
\

\noindent This can be identified with the group of p-adic fractions $$\displaystyle \lim_{i\rightarrow \infty}p^{-i}\mathbbm{Z}\cong \{\frac{n}{p^{j}}\ : m\in \mathbbm{Z}, j\in \mathbbm{N}\}$$ under addition, where the $i^{th}$ copy of $\mathbbm{Z}$ is identified with the cyclic subgroup $p^{-i}\mathbbm{Z}=\{\frac{m}{p^{i}}\ :\ n\in \mathbbm{Z}\}$, and the ``multiplication by p" maps in the diagram correspond to the natural inclusion $\{\frac{m}{p^{i}}\ :\ m\in \mathbbm{Z}\}\hookrightarrow \{\frac{m}{p^{i+1}}\ :\ m\in \mathbbm{Z}\}$ via $\frac{m}{p^{i}}= \frac{pm}{p^{i+1}}$.

Furthermore, note that since $p=1\ \text{mod}\ n$ the map multiplying by $p$ commutes with the quotient homomorphisms $\mathbbm{Z}\rightarrow \mathbbm{Z}/n\mathbbm{Z}$, and thus from the universal property of colimits, we obtain a homomorphism

$$\rho:\text{colim} p^{-i}\mathbbm{Z}\rightarrow \mathbbm{Z}.$$

Furthermore, the kernel of $\rho$ is given by the colimit

\[\begin{tikzcd}
    n\mathbbm{Z} \arrow{r}{p} & n\mathbbm{Z} \arrow{r}{p} & n\mathbbm{Z} \arrow{r}{p} & \dots
 \end{tikzcd}\]

\noindent But note this colimit is again isomorphic to the group of p-adic fraction $\text{colim} p^{-i} \mathbbm{Z}$. We claim that $H^{k}(\text{colim} p^{-i} \mathbbm{Z}, \text{U}(1))=0$ for $k\ge 2$.

Setting $G=\text{colim}\ p^{-i} \mathbbm{Z}$, recall that from the universal coefficients theorem that we have a short exact sequence

\[\begin{tikzcd}
    0 \arrow{r} & \text{Ext}_{\mathbbm{Z}}(H_{k-1}(G, \mathbbm{Z}), \text{U}(1)) \arrow{r} & H^{k}(G, \text{U}(1)) \arrow{r} & \text{Hom}_{\mathbbm{Z}}(H_{k}(G, \mathbbm{Z}), \text{U}(1)) \arrow{r} & 0.
 \end{tikzcd}\]

Since $\text{U}(1)$ is a divisible group, it is an injective $\mathbbm{Z}$ module, hence the $\text{Ext}$ term vanishes.
On the other hand, since homology commutes with directed colimits of groups, we have

$$H_{k}(G, \mathbbm{Z})=H_{k}(\text{colim}\ p^{-i} \mathbbm{Z}, \mathbbm{Z})=\text{colim}\ H_{k}(p^{-i} \mathbbm{Z}, \mathbbm{Z}).$$

\noindent But since $H_{k}(p^{-i} \mathbbm{Z}, \mathbbm{Z})=H_{k}( \mathbbm{Z}, \mathbbm{Z})=0$ for $k\ge 2$, we have $H_{k}(G, \mathbbm{Z})=0$ for $k\ge 2$. Thus $H^{k}(G, \text{U}(1))=0$ for $k\ge 2$.

In particular, this implies for any normalized $3$-cocyle $\omega\in Z^{3}(\mathbbm{Z}/n\mathbbm{Z}, \text{U}(1))$, we have $\rho^{*}(\omega)$ is cohomologically trivial. Furthermore, for any trivialization $c\in C^{2}(G, \text{U}(1))$, the 2-cocyle $c|_{Ker(\rho)}\in Z^{2}(Ker(\rho), \text{U}(1))$ is also cohomologically trivial (since $\text{Ker}(\rho)\cong G)$, thus we may modify $c$ by a 2-coboundary to obtain a trivialization $c^{\prime}$ of $\rho^{*}(\omega)$ with $c^{\prime}|_{\text{Ker}(\rho)}=1$.

Now, consider the action of $G$ on a point. Then since $\text{Ker}(\rho)\cong \text{colim}\ p^{-i} \mathbbm{Z}$, applying Theorem \ref{constructingtheaction} we obtain an $\omega$-anomalous action of $\mathbbm{Z}/n\mathbbm{Z}$ on the reduced group C*-algebra $C^{*}_{r}(\text{colim}\ p^{-i} \mathbbm{Z})$ for all $\omega$. This is a commutative unital C*-algebra, and thus is isomorphic to the algebra of continuous functions on its spectrum $C(X_{p})$. Its spectrum as a C*-algebra coincides with the Pontryagin dual $X_{p}=\widehat{\text{colim}\ p^{-i} \mathbbm{Z}}$. $X_{p}$ is a compact, connected Hausdorff space called a \textbf{p-adic solenoid}. Dualizing the colimit construction described above, it can be realized as a limit of circles, where the maps between circles wrap $p$ times. This space has several unusual properties, as it is not locally connected or path connected. It is thus particularly badly behaved from the point of view of algebraic topology. 

The above construction gives an $\omega$-anomalous action of $\mathbbm{Z}/n\mathbbm{Z}$ on the algebra of functions $C(X_{p})$ of the p-adic solenoid.

\bigskip

\subsection{
Anomalous actions of $\mathbbm{Z}/m\mathbbm{Z}\times \mathbbm{Z}/n\mathbbm{Z}$ on the 2-torus.} \label{actionon2torus}

We consider the case $G:=\mathbbm{Z}/m\mathbbm{Z}\times \mathbbm{Z}/n\mathbbm{Z}$. Elements are represented as pairs $(i,j)$ with $i\in \{0,\dots ,m-1\}$ and $j\in \{0,\dots , n-1\}$, with group operation $\cdot$ given by $(i_{1},i_{2})\cdot(j_{1},j_{2}):=([i_{1}+j_{1}]_{m}, [i_{2}+j_{2}]_{n})$

Then 

$$H^{3}(\mathbbm{Z}/m\mathbbm{Z}\times \mathbbm{Z}/n\mathbbm{Z},\  \text{U}(1)) \cong \mathbbm{Z}/m\mathbbm{Z}\times \mathbbm{Z}/(m,n)\mathbbm{Z}\times \mathbbm{Z}/n\mathbbm{Z},$$

\noindent where $(m,n)$ denotes the greatest common divisor of $m$ and $n$. From \cite{MR3228486}, we have explicit unitary representatives of these cocycles, given for $(a,b,c)\in \mathbbm{Z}/m\mathbbm{Z}\times \mathbbm{Z}/(m,n)\mathbbm{Z}\times \mathbbm{Z}/n\mathbbm{Z}$ by 

$$\omega_{a,b,c}(i,j,k):=\text{e}^{\frac{ 2\pi i}{m} a i_{1}\floor{\frac{j_{1}+k_{1}}{m}}} \text{e}^{\frac{2\pi i}{n}( b i_{2}\floor{\frac{j_{1}+k_{1}}{m}}+ci_{2}\floor{\frac{j_{2}+k_{2}}{n}} )},$$

\noindent where $i=(i_{1},i_{2}), j=(j_{1},j_{2}), k=(k_{1},k_{2})$. 

\medskip

\noindent \textbf{Warning}: in the above formula, we do not reduce the sums modulo $n$ (e.g. $i_{1}+j_1$ \text{mod} $n$). We use the notation $[i]_{n}$ to indicate reduction \text{mod} $n$ in this section of the paper.

\medskip

Our goal is to obtain an explicit $\omega$-anomalous action of $\mathbbm{Z}/m\mathbbm{Z}\times \mathbbm{Z}/n\mathbbm{Z}$ on the 2-torus $\mathbb{T}^{2}$ for \textit{any} 3-cocycle $\omega\in H^{3}(\mathbbm{Z}/m\mathbbm{Z}\times \mathbbm{Z}/n\mathbbm{Z}, \text{U}(1))$. Consider the $*$-algebra of Laurent polynomials in two commuting unitary variables $\mathbbm{C}[z,z^{-1},w, w^{-1}]$. Then there is a natural embedding $\mathbbm{C}[z,z^{-1},w, w^{-1}]\rightarrow C(\mathbb{T}^{2})$ whose image is dense.

We have a homomorphism $\alpha: \mathbbm{Z}/m\mathbbm{Z}\times \mathbbm{Z}/n\mathbbm{Z}\rightarrow \text{Aut}(\mathbbm{C}[z,z^{-1},w, w^{-1}])$ given by

$$\alpha_{(j,k)}(f)(z,w):=f(\text{e}^{\frac{2\pi i j}{m}}z,\text{e}^{\frac{2\pi i k}{n}}w). $$

\noindent Geometrically this simply rotates by $\frac{2\pi j}{m}$ clockwise on the first factor and $\frac{2\pi k}{n}$ on the second factor in the decomposition $\mathbb{T}^{2}=S^{1}\times S^{1}$. Clearly this extends to an automorphism of $C(\mathbb{T}^{2})$.

Now, to get an $\omega$-anomalous action of $\mathbbm{Z}/m\mathbbm{Z}\times \mathbbm{Z}/n\mathbbm{Z}$ associated to this homomorphism, we need to find a trivialization of $\omega$ \textit{as a cocycle in} $Z^{3}(G, C(\mathbb{T}^{2},\text{U}(1)))$. Using Definition \ref{Explicitdefn}, we need to find unitary functions

$$\mu_{i,j}(z,w)\in C(\mathbb{T}^{2}, \text{U}(1))$$

\noindent for $i=(i_1, i_2), j=(j_{1},j_{2})\in \mathbbm{Z}/m\mathbbm{Z}\times \mathbbm{Z}/n\mathbbm{Z}$, satisfying the equation 

\begin{equation}\label{cocycleequation}
\omega(i,j,k) \mu_{i\cdot j, k}(z,w)\ \mu_{i,j}(z,w)= \mu_{i,j\cdot k}(z,w)\  \alpha_{i}(\mu_{j,k})(z,w)
\end{equation}

\noindent for $i,j,k\in \mathbbm{Z}/m\mathbbm{Z}\times \mathbbm{Z}/n\mathbbm{Z} $.

Fix a $3$-cocycle $\omega_{a,b,c}$ as described above and consider the function

$$\mu_{i,j}(z,w):=(\text{e}^{\frac{2\pi i}{m}} z)^{a \floor{\frac{i_{1}+j_{1}}{m}}} (\text{e}^{\frac{2\pi i}{n}} w)^{( b \floor{\frac{i_{1}+j_{1}}{m}}+c\floor{\frac{i_{2}+j_{2}}{n}} )}.$$

\begin{prop}
$\mu_{i,j}(z,w)$ defined above satisfies the cocycle equation (\ref{cocycleequation}) for $\omega_{a,b,c}$.
\end{prop}

\begin{proof}
First note that we have the following formula, of which we make repeated use in the following computations \cite[Lemma 2.1]{MR3228486}:

\begin{equation}\label{floorlemma}
\floor{\frac{i+[j]_{m}}{m}}=\floor{\frac{i+j}{m}}-\floor{\frac{j}{m}}.
\end{equation}

\noindent We compute
\begin{align*}
&\mu_{i\cdot j, k}(z,w)\ \mu_{i,j}(z,w)\\
=&\left(\text{e}^{\frac{2\pi i}{m}} z\right)^{a (\floor{\frac{[i_{1}+j_{1}]_{n}+k_{1}}{m}}+ \floor{\frac{i_{1}+j_{1}}{m}})} \left(\text{e}^{\frac{2\pi i}{n}} w\right)^{ b(\floor{\frac{[i_{1}+j_{1}]_{m}+k_{1}}{m}} +\floor{\frac{i_{1}+j_{1}}{m}})+c(\floor{\frac{[i_{2}+j_{2}]_{n}+k_{2}}{n}} +\floor{\frac{i_{2}+j_{2}}{n}})}\\
=&\left(\text{e}^{\frac{2\pi i}{m}} z\right)^{a \floor{\frac{i_{1}+j_{1}+k_{1}}{m}}} \left(\text{e}^{\frac{2\pi i}{n}} w \right)^{ b\floor{\frac{i_{1}+j_{1}+k_{1}}{m}}+c\floor{\frac{i_{2}+j_{2}+k_{2}}{n}}}. 
\end{align*}

\noindent Now, we compute the right hand side of the cocycle equation. We have

$$\mu_{i,j\cdot k}(z,w)=\left(\text{e}^{\frac{2\pi i}{m}} z\right)^{a \floor{\frac{i_{1}+[j_{1}+k_{1}]_{m}}{m}}} \left(\text{e}^{\frac{2\pi i}{n}} w\right)^{b \floor{\frac{i_{1}+[j_{1}+k_{1}]_{m}}{m}}+c\floor{\frac{i_{2}+[j_{2}+k_{2}]_{n}}{n}} },$$

$$\alpha_{i}(\mu_{j,k})(z,w)= \left(\text{e}^{\frac{2\pi i(1+i_{1})}{m}} z \right)^{a \floor{\frac{j_{1}+k_{1}}{m}}} \left(\text{e}^{\frac{2\pi i(1+i_{2})}{n}} w\right)^{b \floor{\frac{j_{1}+k_{1}}{m}}+c\floor{\frac{j_{2}+k_{2}}{n}} }.$$

\noindent
By rearranging the exponents and applying equation \ref{floorlemma}, we get

\begin{align*}
&\mu_{i,j\cdot k}(z,w) \alpha_{i}(\mu_{j,k})(z,w)\\
=&\left(\text{e}^{\frac{ 2\pi i}{m} a i_{1}\floor{\frac{j_{1}+k_{1}}{m}}}\right) \left(\text{e}^{\frac{2\pi i}{n}( b i_{2}\floor{\frac{j_{1}+k_{1}}{m}}+ci_{2}\floor{\frac{j_{2}+k_{2}}{n}} )}\right) \left(\text{e}^{\frac{2\pi i}{m}} z\right)^{a \floor{\frac{i_{1}+j_{1}+k_{1}}{m}}} \left(\text{e}^{\frac{2\pi i}{n}} w\right)^{ b\floor{\frac{i_{1}+j_{1}+k_{1}}{m}}+c\floor{\frac{i_{2}+j_{2}+k_{2}}{n}}} \\
=&\omega(i,j,k) \mu_{i\cdot j, k}(z,w)\ \mu_{i,j}(z,w).\\
\end{align*}

\noindent as desired.

\end{proof}

\begin{rem}
We note that setting $n=1$, the action restricts to the first factor, and we obtain the $\omega$-anomalous action of $\mathbbm{Z}/m \mathbbm{Z}$ on $S^{1}$ for arbitrary $\omega\in H^{3}(G, \text{U}(1))$. It is not hard to show these agree with the actions constructed in the previous section.
\end{rem}

\subsection{A no-go theorem.}
Now that we have seen many examples, we include a proof of Theorem \ref{noactionsonmanifolds} from the introduction.

\begin{proof}[Proof of Theorem \ref{noactionsonmanifolds}]

It is easy to see from Gelfand duality that for any compact Hausdorff space $X$, $$\underline{\text{Aut}}(C(X))\cong \text{2-Gr}(\text{Homeo}(X), C(X, \text{U}(1)), 1),$$ with the obvious action of $\text{Homeo}(X)$ on $C(X, \text{U}(1))$, and trivial associator 3-cocycle. Thus an $\omega$-anomalous G-action on $C(X)$ can be built if and only if we can find a homomorphism $\rho:G\rightarrow \text{Homeo}(X)$ so that $i_{*}(\omega)$ is trivial in $H^{3}(G, C(X, \text{U}(1))$, where $i_{*}$ is the pushforward of the inclusion of coefficient modules $\text{U}(1)\hookrightarrow C(X, \text{U}(1))$. However, we will show that the hypothesis imply that the map $i_{*}$ is an isomorphism on cohomology.

We have an exact sequence of $G$-modules $0\rightarrow C(X,\mathbbm{Z})\rightarrow C(X, \mathbbm{R})\rightarrow C(X, \text{U}(1))$, where $\mathbbm{Z}$ is considered as an additive discrete group, $\mathbbm{R}$ is considered as an additive topological group, and $\text{U}(1)$ is a multiplicative topological subgroup of $\mathbbm{C}^{\times}$. The map $C(X,\mathbbm{Z})\rightarrow C(X, \mathbbm{R})$ is induced from the inclusion the inclusion $\mathbbm{Z}\rightarrow \mathbbm{R}$ while the map $C(X,\mathbbm{R})\rightarrow C(X,\text{U}(1))$ is given by $f\mapsto \text{exp}(2\pi i f)$. In general this map fails to be surjective. But since $H^{1}(X, \mathbbm{Z})=0$, by the universal coefficients theorem $\text{Hom}(H_{1}(X,\mathbbm{Z}), \mathbbm{Z})=0$. But $H_{1}(X,\mathbbm{Z})$ is the abelianization of $\pi_{1}(X)$ and by the universal property of abelianization, we have $\text{Hom}(\pi_{1}(X), \mathbbm{Z})=\text{Hom}(H_{1}(X,\mathbbm{Z}), \mathbbm{Z})=0$. Thus for any map $g\in C(X, \text{U}(1))$ the induced map $g_{*}:\pi_{1}(X)\rightarrow \pi_{1}(\text{U}(1))=\mathbbm{Z}$ must be trivial. Then since the map $\text{exp}(2\pi i \cdot):\mathbbm{R}\rightarrow \text{U}(1)$ is a covering, by standard results in algebraic topology (e.g. \cite[Proposition 1.33]{MR1867354}), since $X$ is connected and locally path connected by assumption, there exists a lift $\tilde{g}:X\rightarrow \mathbbm{R}$. In particular, this implies $C(X,\mathbbm{R})\rightarrow C(X,\text{U}(1))$ is surjective so the above sequence can be extended to a short exact sequence. In fact, we have a morphism of short exact sequences

\[\begin{tikzcd}
    0 \arrow{r} & C(X, \mathbbm{Z})\arrow{r} & C(X, \mathbbm{R})\arrow{r} & C(X, \text{U}(1))\arrow{r} & 0\\
    0 \arrow{r} & \mathbbm{Z}\arrow{u}\arrow{r} &  \mathbbm{R}\arrow{u}\arrow{r} &  \text{U}(1)\arrow{u}\arrow{r} & 0
 \end{tikzcd}\]
 
\noindent where the vertical homomorphism sends elements of the group to the constant function. Therefore, we obtain a morphism of long exact sequences in cohomology

\[\begin{tikzcd}
     \cdots  H^{i}(G,C(X, \mathbbm{R}))\arrow{r} & H^{i}(G,C(X, \text{U}(1)))\arrow{r} & H^{i+1}(G, C(X,\mathbbm{Z}))\arrow{r} &  \cdots\\
    \cdots H^{i}(G,\mathbbm{R})\arrow{u}\arrow{r} & H^{i}(G, \text{U}(1))\arrow{u}\arrow{r} & H^{i+1}(G, \mathbbm{Z})\arrow{r}\arrow{u} & \cdots 
 \end{tikzcd}\]
 
\noindent where the vertical maps in cohomology are simply the pushforwards of the morphisms between the coefficient groups (this follows from ``naturality" of the connecting homomorphism).
But since $\mathbbm{R}$ and $C(X,\mathbbm{R})$ are vector spaces and $G$ is finite, $|G|$ is invertible hence these module are cohomologically trivial. Thus we have the commuting diagram with exact rows

\[\begin{tikzcd}
     0\arrow{r} & H^{3}(G,C(X, \text{U}(1)))\arrow{r} & H^{4}(G, C(X,\mathbbm{Z}))\arrow{r} & 0 \\
    0\arrow{r} & H^{3}(G, \text{U}(1))\arrow{u}\arrow{r} & H^{4}(G, \mathbbm{Z})\arrow{r}\arrow{u} & 0
 \end{tikzcd}\]

Thus the two horizontal arrows are isomorphisms. Since $X$ is connected, the right vertical arrow is also an isomorphism. This implies the left vertical arrow $H^{3}(G, \text{U}(1))\rightarrow H^{3}(G, C(X,\text{U}(1)) )$ is an isomorphism as well. This concludes the proof of the first part of the theorem.

For the second part, note that by \cite{MR589649}, there is a short exact sequence $1\rightarrow \text{Pic}(C(X))\rightarrow \text{Out}(C(X)\otimes \mathcal{K})\rightarrow \text{Homeo}(X)\rightarrow 1$, where where $\text{Pic}(X)$ denotes the group of line bundles on $X$ and $\text{Homeo}(X)$ is the group of homeomorphisms acting on $C(X)$ by *-automorphisms via Gelfand duality and extending to $C(X)\otimes \mathcal{K}$. This action induces a natural splitting, hence $\text{Out}(C(X)\otimes \mathcal{K})\cong \text{Pic}(X)\rtimes \text{Homeo}(X)$. But this splitting actually extends to a fully faithful monoidal functor between the 2-groups $\underline{\text{Aut}}(C(X))\hookrightarrow \underline{\text{Aut}}(C(X)\otimes \mathcal{K})$, and if $\text{Pic}(X)=0$, it is essentially surjective and is thus an equivalence of 2-groups, and we can apply the previous theorem. Note that if $X$ is homotopy equivalent to a CW-complex, then $H^{2}(X,\mathbbm{Z})\cong \text{Pic}(X)$ (e.g. \cite[Proposition 3.10]{Hatcher}).

\end{proof}

\begin{cor}\label{non-trivial-lifting}
For every $n\ge 4$, there exists infinitely many closed connected $n$-manifolds $M$ such that the Postnikov associator (i.e. lifting obstruction) $[\omega_{M}]\in H^{3}(\text{Out}(C(M)\otimes \mathcal{K}), C(M, \text{U}(1)) )$ is non-trivial.
\end{cor}

\begin{proof}
Let $n\ge 4$, $G$ a finite group, and $\omega\in Z^{3}(G, \text{U}(1))$. Let $M$ be an $n$-manifold with $\pi_{1}(M)$ finite abelian such that there exists an $\omega$-anomalous $G$ action on $C(M)\otimes \mathcal{K}$ (which always exist by Example \ref{manifoldspi1finite}). Let $\rho: G\rightarrow \text{Out}(C(M)\otimes \mathcal{K})$ be the corresponding homomorphism. By the proof of the previous theorem, $i_{*}: H^{3}(\rho(G), \text{U}(1))\rightarrow H^{3}(\rho(G), C(M, \text{U}(1)))$ is an isomorphism. This implies $[\omega_{M}]|_{\rho(G)}$ is cohomologous to a scalar valued $3$-cocycle which is trivial in $H^{3}(\rho(G), \text{U}(1))$ if and only if $[\omega_{M}]|_{\rho(G)}$ is trivial in $H^{3}(\rho(G), C(M, \text{U}(1)))$. But this is non-trivial, since its image under $\rho^{*}$ is $\omega\in H^{3}(G, \text{U}(1))$, which is non-trivial by assumption.
\end{proof}

\section{Anomalies in coarse geometry.}\label{anomaliesincoarsegeom}

Coarse geometry is concerned with the study of the large scale properties of geometric spaces. The noncommutative perspective we consider here initially arose from the study of index theory for elliptic operators on non-compact Riemannian manifolds \cite{MR1147350}. C*-algebras arising in coarse geometry have recently found applications in the study of topological band structures of single particle systems \cite{MR3927086}. This latter perspective motivates our study, and we will say more about this at the end of the section.

Here we recall some basic notions from the theory of coarse metric geometry. Let $X$ be a discrete metric space. $X$ has \textit{bounded geometry} if for every $R>0$ there exists an $N$ such that the $R$-ball of any point has at most $N$ elements. We will primarily be concerned with bounded geometry metric spaces. For two bounded geometry metric spaces $X,Y$, a function $f:X\rightarrow Y$ is called \textit{coarse} if for all $R$, there exists an $S$ such that $d(x,y)\le R$ implies $d(f(x),f(y))\le S$. Two coarse maps $f,g: X\rightarrow Y$ are \textit{close} if there exists some $S$ such that $d(f(x),g(x))\le S$ for all $X\in X$. A \textit{coarse equivalence} consists of coarse maps $f:X\rightarrow Y$ and $f^{-1}:Y\rightarrow X$ such that $f\circ f^{-1}$ and $f^{-1}\circ f$ are close to the identity on $Y$ and $X$ respectively. Large scale properties of $X$ and $Y$ are invariant under coarse equivalence. Natural examples arise from geometric group theory. Given a discrete group $G$ and a finite presentation, we can endow the set $G$ with a metric space structure using the word metric. It is easy to show that for any two presentations on the same group, the identity map is a coarse equivalence.

Recall that a metric space with bounded geometry has Yu's \textit{Property A} if for every $\epsilon>0$ and $R>0$, there exists a collection of subsets $\{A_{x}\subseteq X\times \mathbbm{N}\}_{x\in X}$ and a constant $S$ such that

\begin{itemize}
\item
$A_{x}\subseteq B_{S}(x)\times \mathbbm{N}$.
\item 
$\frac{|A_{x}\Delta A_{y}|}{|A_{x}\cap A_{y}|}\le \epsilon$ when $d(x,y)\le R$.
\end{itemize}

\noindent Property $A$ is a geometric version of amenability. Given a finitely generated group $\Gamma$ with word-length metric, this is a metric space with bounded geometry. It has property $A$ if and only if the reduced group C*-algebra is exact \cite{MR1876896, MR1876897}.

Now, let $H$ be a separable infinite dimensional Hilbert space, and consider the Hilbert space $\ell^{2}(X, H)$. For $x\in X$, let $P_{x}$ denote the projection onto functions $\eta\in \ell^{2}(X, H)$ supported at the point $x$, which is naturally isomorphic to $H$. An operator $T$ on this space is called \textit{locally compact} if for every $x,y\in X$, the restriction $P_{y}TP_{x}$ is a compact operator. $T$ has \textit{finite propagation} if there exists an $R>0$ such that $P_{y}TP_{x}\ne 0$ implies $d(x,y)<R$. The \textit{Roe algebra} $C^{*}(X)$ is defined as the norm closure of the locally compact operators with finite propagation in $B(\ell^{2}(X, H))$. Their algebraic and analytic structure captures important properties of the underlying large scale structure. Indeed, if $X$ is coarsely equivalent to $Y$, then $C^{*}(X)\cong C^{*}(Y)$. Conversely in the presence of property $A$, these C*-algebras remember the large scale structure completely: if $X$ and $Y$ have property $A$ and $C^{*}(X)\cong C^{*}(Y)$, then $X$ is coarsely equivalent to $Y$ \cite{MR3116573}.

Consider the group of coarse auto-equivalences of a bounded geometry metric space $X$, taken up to the equivalence relation of closeness. We denote this group $\text{Coa}(X)$. Then we have a natural homomorphism $\text{Coa}(X)\rightarrow \text{Out}(C^{*}(X))$. In fact, when $X$ has property $A$, this is an isomorphism of groups \cite{BV2020}. Given the rich geometric structure of $\text{Out}(C^{*}(X))$, any anomalous symmetries of a Roe C*-algebra would be very interesting since they should have a coarse-geometric interpretation. Unfortunately, anomalous symmetries of Roe algebras do not exist. 

To state the following theorem, note that $\text{UZM}(C^{*}(X))\cong \text{U}(1)$, so the Postnikov assocatior data gives us a class $[\omega_{X}]\in H^{3}(\text{Out}(C^{*}(X)), \text{U}(1))$. 

\begin{prop}[Theorem \ref{noactiononcoarse}]
Let $\omega_{X}\in Z^{3}(\text{Out}(C^{*}(X)), \text{U}(1)))$ be a representative Postnikov cocycle. Then $[\omega_{X}]$ is trivial in $H^{3}(\text{Out}(C^{*}(X)), \text{U}(1))$. As a consequence, there are no anomalous actions on $C^{*}(X)$.
\end{prop}

\begin{proof}
Let $\mathcal{K}$ denote the algebra of compact operators on a separable Hilbert space. Then there is a natural inclusion $\mathcal{K}\hookrightarrow C^{*}(X)$ via identifying $\mathcal{K}$ with the compact operators on $\ell^{2}(X, H)$. This ideal is the unique minimal ideal by \cite[Corollary 2.3]{MR1874489}

Since $\mathcal{K}$ is the unique minimal ideal for any automorphism $\alpha\in \underline{\text{Aut}}(C^{*}(X))$ $\alpha(\mathcal{K})=\mathcal{K}$, and thus $\alpha$ restricts to an auto-equivalence of $\mathcal{K}$.
But note that the multiplier algebra of $C^{*}(X)$ is naturally identified inside $B(\ell^{2}(X, H))=M(\mathcal{K})$ as the closure of the algebra of band dominated operators (these are the operators with finite propagation but not necessarily locally compact) (\cite[Theorem 4.1]{BV2020}). Using this identification, for any unitary $u\in M(C^{*}(X))$ that intertwines automorphisms $\alpha$ and $\beta$, $u\in M(\mathcal{K})$ intertwines $\alpha|_{\mathcal{K}}$ and $\beta|_{\mathcal{K}}$. Thus restriction gives a natural monoidal functor from $\underline{\text{Aut}}(C^{*}(X)\rightarrow \underline{\text{Aut}}(\mathcal{K})$. But since all automorphisms of $\mathcal{K}$ are inner, $\pi_{1}(\underline{\text{Aut}}(\mathcal{K}))$ is trivial, and so $\underline{\text{Aut}}(\mathcal{K})\cong \text{2-Gr}(1,\text{U}(1),1)$. 

Thus we have a $\text{U}(1)$-linear monoidal functor $\underline{\text{Aut}}(C^{*}(X))\cong \text{2-Gr}(\text{Out}(C^{*}(X)), \text{U}(1), \omega_{X})\rightarrow \underline{\text{Aut}}(\mathcal{K})\cong \text{2-Gr}(1,\text{U}(1),1)$. But this gives a trivialization of $\omega_{X}$.

\end{proof}

The previous theorem is somewhat disappointing. But as the proof demonstrates, the problem seems to be with the existence of a unique minimal essential ideal of compact operators $\mathcal{K}\le \text{C*}(X)$. It is natural, then, to ask whether there can be anomalous symmetries of the algebra $\text{C*}(X)/\mathcal{K}$?

The answer is yes, and we can use our technology to build examples. Let $G$ be a finite group, and suppose that we have an action of $G$ on $X$ by coarse auto-equivalences. We will assume, for convenience, that each $g$ acts bijectively on $X$, and that the group relations are satisfied on the nose (not only up to closeness). We say the action is \textit{coarsely discontinuous} if for all $R$, there exists a bounded subset $K\subseteq X$ such that $d(x,g\cdot x)\ge R$ for all $x\notin K$.

We define a metric on the orbit space $X/G$ using the Hausdorff metric for subsets of a metric space, namely

$$\displaystyle d_{G}([x],[y])=\max\{\  \max_{h\in G}\ \min_{g\in G}\ d(hx,gy),\ \ \  \max_{h\in G}\ \min_{g\in G} d(hy,gx)\ \}.$$

\noindent This metric makes $X/G$ into a metric space with bounded geometry. 

Recall that $C^{*}(X)$ contains a copy of $\mathcal{K}$, the algebra of compact operators on a separable Hilbert space realized as $\mathcal{K}(\ell^{2}(X, H))\subseteq C^{*}(X)$. This is the unique minimal ideal. We have the following theorem: 

\begin{prop}\cite[Corollary 5.10]{MR4040015} 
Let $X$ be a bounded geometry metric space with property $A$, and $G$ a finite group acting coarsely discontinuously on $X$. Then there is an isomorphism $C^{*}(X)/\mathcal{K}\rtimes G\cong C^{*}(X/G)/\mathcal{K}$.

\end{prop}

If $X$ does not have property $A$, the above theorem still holds upon replacing the Roe algebra $C^{*}(X)$ with a universal version $C^{*}_{max}(X)$. However, the examples of interest, particularly those applicable in physics, have property $A$, so we will stick to the Roe algebra. Applying our machinery, we obtain the following:

\begin{cor}[Theorem \ref{actiononPropertyA}]
Let $G$ be a finite group. Then for any anomaly $\omega\in Z^{3}(G, \text{U}(1))$, there exists a bounded geometry metric space $X$ with property $A$ and an $\omega$-anomalous action of $G$ on the Roe corona $C^{*}(X)/\mathcal{K}$. 
\end{cor}

\begin{proof}
Let $\tilde{G}$ be a finite group and $\rho:\tilde{G}\rightarrow G$ a surjective homomorphism together with a trivialization $c$ of $\rho^{*}(\omega)$ such that $c|_{K}=1$, where $K:=\text{Ker}(\rho)$. Let $n=|\tilde{G}|$, and consider the free group $\mathbbm{F}_{n}$. This group together with the standard word metric is a bounded geometry metric space with Property (A). Let us label the generators by elements of $\tilde{G}$, $\{g_{a}\}_{a\in \tilde{G}}$. Then we have an action by $\tilde{G}$ on $\mathbbm{F}_{n}$ by automorphisms by permuting the generators via 

$$b\cdot g_{a}:=g_{ba}.$$

To check this is coarsely discontinuous for each $b$, let $w=w_{a_1}\dots w_{a_k}$ be a reduced word expression for $w\in \mathbbm{F}_{n}$. Then 

$$d(w,b\cdot w)=\ell(w^{-1}_{ba_{k}}\dots w^{-1}_{ba_{1}}w_{a_1}\dots w_{ba_k}).$$

But since $ba_{1}\ne a_{1}$ unless $b=1$, $d(w,b\cdot w)=2k$. Therefore, for each $R\ge 0$, pick since $k>\frac{R}{2}$. If $w$ is not in the ball $B_{k}(1)$ of radius $k$ around the identity, $d(w,b\cdot w)>R$, and thus the action of $\tilde{G}$ is coarsely discontinuous. Then we can apply Theorem \ref{constructingtheaction} to obtain an $\omega$-anomalous action on the Roe corona of $X/K$.
\end{proof}

\begin{ex}
From the discussion at the end of Section \ref{cohomologylemma}, if we let $\rho:\mathbbm{Z}/4\mathbbm{Z}\rightarrow \mathbbm{Z}/2\mathbbm{Z}$, $\omega\in Z^{3}(\mathbbm{Z}/2\mathbbm{Z}, \text{U}(1))$, then we can find $c\in C^{2}(\mathbbm{Z}/4\mathbbm{Z}, \text{U}(1))$ with $dc=\omega$ and $c|_{\text{Ker}(\rho)}=1$. 

Consider the $\mathbbm{Z}/4\mathbbm{Z}$ action on the lattice $\mathbbm{Z}^{2}$ (with the usual metric), generated by $(x,y)\mapsto (-y,x)$. It is easy to see that this action is coarsely discontinuous, and thus we have an $\omega$-anomalous action of $\mathbbm{Z}/2\mathbbm{Z}$ on the quotient of $\mathbbm{Z}^{2}$ by the $\mathbbm{Z}/2\mathbbm{Z}$ action $(x,y)\mapsto (-x,-y)$. This can be realized as a set by the upper half plane intersected $\mathbbm{Z}^{2}$, minus the negative part of the $x$-axis, but with the metric for $d(a,b)=\text{min}\{d(a,b),d(a,-b)\}$.
\end{ex}

We also note that the proof of \cite[Corollary 5.10]{MR4040015} can be easily adapted to replace the Roe algebra $C^{*}(X)$ by the stabilized uniform Roe algebra $C^{*}_{u}(X)\otimes \mathcal{K}$. Thus the constructions of this section will also produce anomalous actions in the stabilized uniform Roe algebra.

\subsection{Actions on wedges of spaces.}

In this section we will construct explicit anomalous actions of an arbitrary (finite) group $G$ on Roe coronas of a wedge of spaces over a point. This construction works for uniform Roe coronas as well. In some sense, these actions are ``less interesting" in general than the ones we can build with the twisted crossed product machinery. This is because they are asymptotically discrete, which will allow us to build them from an the usual data of an $\omega$-anomalous action on a finite set.

For simplicity, we will take $G$ acting on itself as our finite set model. Let $Y$ be any (infinite) metric space with bounded geometry and let $d$ denote the metric. Take the disjoint union of $|G|$ copies of $Y$. Fix a point $y_{0}\in Y$ and glue all the copies of $Y$ together at this point. We denote the subset corresponding to the $g^{th}$ component as $Y_{g}$. This space has a metric $\tilde{d}$, which restricted to each copy of $Y$ yields $d$, but for $x,y$ in distinct copies of $Y$, $\tilde{d}(x,y)=d(x,y_{0})+d(y_{0},x)$ for $x\in Y_{g},\ y\in Y_{h}$ and $g\ne h$. This defines a metric space which we denote $X$. This has an action of $G$ by isometries $g(y_{h})=y_{gh}$ for any point $y_{h}\in Y_{h}-\{y_{0}\}$, and $g(y_{0})=y_{0}$. Each group element has precisely one fixed point, $y_{0}$.

The $G$-action is implemented by unitaries $U_{g}\in \ell^{2}(X, H)$, defined by $U_{g}(\eta(x))=\eta(g^{-1}(x))$, and we have $U_{g}U_{h}=U_{gh}$. Furthermore, the automorphism $\alpha_{g}(a)=U_{g}aU^{*}_{g}$ preserves the Roe C*-algebra $C^{*}(X)\le \text{B}(\ell^{2}(X, H))$

For each $g,h\in G$ and $\omega\in Z^{3}(G, \text{U}(1))$, define

\begin{equation*}
    m_{g,h}(\eta(x)) = \begin{cases}
               \omega(g,h,h^{-1}g^{-1}k)^{-1}\eta(x) & x\in Y_{k}-\{y_{0}\}\\
               1               & x=y_{0}
           \end{cases}.
\end{equation*}

\noindent Since $m_{g,h}$ is diagonal, it is band dominated hence lives in the multiplier algebra of $C^{*}(X)$. Thus $[m_{g,h}]_{\mathcal{K}}\in B(\ell^{2}(X, H))/\mathcal{K}$ can be identified inside the multiplier algebra of the Roe corona $C^{*}(X)/\mathcal{K}$. 

\begin{prop}
The assignment $g\mapsto \alpha_{g}$ together with the collection $[m_{g,h}]_{\mathcal{K}}$ defines an $\omega$-anomalous action of $G$ on $C^{*}(X)/\mathcal{K}$.
\end{prop}

\begin{proof}
To begin, we need to show $$[m_{g,h}]_{\mathcal{K}}\ \alpha_{g}\circ \alpha_{h}(\ [a]_{\mathcal{K}}\ )=\alpha_{gh}(\ [a]_{\mathcal{K}}\ )[m_{g,h}]_{\mathcal{K}},$$

But since we already have $\alpha_{g}\circ \alpha_{h}=\alpha_{gh}$ on $C^{*}(X)$, it suffices to show the unitaries $[m_{g,h}]_{\mathcal{K}}$ are in the center of the multiplier algebra of $C^{*}(X)/\mathcal{K}$. In other words, we claim $m_{g,h}$ commutes with any operator $a\in C^{*}(X)$ up to a compact operator. It suffices to check this for locally compact operators $a$ with finite propagation $R$.  

Define the subset $X_{R}$ as the ball of radius $R+1$ in $X$ around the point $y_{0}$, which is finite by our bounded geometry assumption on $Y$. Let $P_{R}$ denote the projection in $B(\ell^{2}(X, H))$ onto functions supported in $X_{R}$. Let $Q_{R}$ denote its-complement, i.e. the projection onto functions supported in $X-X_{R}$. 

Note that $a$ must send functions $\eta\in \ell^{2}(X, H)$ supported in $Y_{g}\cap (X-X_{R})$ to functions that again supported in $Y_{g}-\{y_{0}\}$ for each $g$, since for $x\in Y_{g}\cap (X-X_{R})$ and $y\in Y_{h}\cap (X-X_{R})$, $d(x,y)>2R$, and $d(x,y_{0})\ge R+1$. Thus since $m_{g,h}$ acts by a constant scalar on any function $\eta \in \ell^{2}(X, H)$ supported in $Y_{g}-\{y_{0}\}$ for any $g$, we have $\left[ a,m_{g,h} \right] Q_{R}=0$.

Thus $$\left[a,m_{g,h}\right]=\left[a,m_{g,h}\right]Q_{R}+\left[a,m_{g,h}\right]P_{R}=\left[a,m_{g,h}\right]P_{R}.$$

\noindent But since the set $X_{R}$ is finite and $\left[a,m_{g,h}\right]$ is locally compact and finite propagation, $\left[a,m_{g,h}\right]P_{R}$ is compact as $X$ has bounded geometry. This gives us the claim.

Now we have to to show

$$\omega(g,h,k)[m_{gh,k}]_{\mathcal{K}}[m_{g,h}]_{\mathcal{K}}=[m_{g,hk}]_{\mathcal{K}}\alpha_{g}([m_{h,k}]_{\mathcal{K}}).$$

\noindent We need this equations to hold \textit{in the multiplier algebra} of $C^{*}(X)/\mathcal{K}$. For any $\eta\in \ell^{2}(X, H)$ and for any $x\in Y_{ghkl}- \{y_{0}\}$, we have

$$\omega(g,h,k)m_{gh,k}m_{g,h}(\eta(x))=\omega(g,h,k)\omega(gh,k,l)^{-1}\omega(g,h,kl)^{-1}\eta(x),$$

\noindent and

$$m_{g,hk}\alpha_{g}(m_{h,k})(\eta(x))=\omega(g,hk,l)^{-1}\omega(h,k,l)^{-1}\eta(x).$$

\noindent But these are equal by the definition of $3$-cocyle. Let $P_{0}$ be the projection in $\ell^{2}(X, H)$ to the space of functions supported on $y_{0}$, and $Q_{0}$ its complement. Then we have for any $a\in C^{*}(X)$

$$\omega(g,h,k)m_{gh,k}m_{g,h}a=\omega(g,h,k)m_{gh,k}m_{g,h}P_{0}\ a+\omega(g,h,k)m_{gh,k}m_{g,h}Q_{0}\ a, $$

\noindent and

$$\omega(g,hk,l)^{-1}\omega(h,k,l)^{-1}a=\omega(g,hk,l)^{-1}\omega(h,k,l)^{-1}P_{0}a+\omega(g,hk,l)^{-1}\omega(h,k,l)^{-1}Q_{0}\ a,$$

\noindent hence
\begin{align*}
 &(\omega(g,h,k)m_{gh,k}m_{g,h}-\omega(g,hk,l)^{-1}\omega(h,k,l)^{-1})a\\
 =&(\omega(g,h,k)m_{gh,k}m_{g,h}-\omega(g,hk,l)^{-1}\omega(h,k,l)^{-1})P_{0}a
\end{align*}

\noindent is compact. This gives us the desired equality in the multiplier algebra of $C^{*}(X)/\mathcal{K}$.

\end{proof}

\begin{rem}
We remark that the proofs in the above construction were not sensitive to the dimension of the Hilbert space $H$. Thus we could have selected $H=\mathbbm{C}$, in which case the Roe algebra in the above statement is replaced by the uniform Roe algebra $C^{*}_{u}(X)$. Thus it is also the case that every anomaly of a finite group is realized via an action on some uniform Roe corona as well.
\end{rem}

\bibliographystyle{amsalpha}
{\footnotesize{
\bibliography{bibliography}
}}

\end{document}